\documentclass[12pt]{article}

\usepackage{amssymb}
\usepackage{amsthm}
\usepackage{amsmath}
\usepackage{graphicx}
\usepackage{fullpage}
\usepackage{color}
\usepackage{enumerate}
 \numberwithin{equation}{section}
\usepackage{booktabs}
\allowdisplaybreaks
\newtheorem{innercustomthm}{Theorem}
\newenvironment{customthm}[1]
  {\renewcommand\theinnercustomthm{#1}\innercustomthm}
  {\endinnercustomthm}
\usepackage[pdftex]{hyperref}
 \usepackage{mathtools}
 \mathtoolsset{showonlyrefs}
\theoremstyle{plain}
\newtheorem{thm}{Theorem}[section]
\newtheorem{cor}[thm]{Corollary}

\newtheorem{lem}[thm]{Lemma}
\newtheorem{prop}[thm]{Proposition}

\theoremstyle{definition}
\newtheorem{defn}[thm]{Definition}

\theoremstyle{remark}

\newtheorem{rem}[thm]{Remark}

\newcommand{\N}{\mathbb{N}}
\newcommand{\R}{\mathbb{R}}

\newcommand{\bp}{\begin{proof}[\ensuremath{\mathbf{Proof}}]}
\newcommand{\bs}{\begin{proof}[\ensuremath{\mathbf{Solution}}]}
\newcommand{\ep}{\end{proof}}

\begin{document}

\title{Large time behavior of solutions of Trudinger's equation}

\author{Ryan Hynd\footnote{Department of Mathematics, MIT.  Partially supported by NSF grant DMS-1554130 and an MLK visiting professorship.}\; and Erik Lindgren\footnote{Department of Mathematics, KTH. Supported by the Swedish Research Council, grant no. 2012-3124.}}  

\maketitle

\begin{abstract}
We study the large time behavior of solutions $v:\Omega\times(0,\infty)\rightarrow \R$ of
the PDE $\partial_t(|v|^{p-2}v)=\Delta_pv.$ We show that $e^{\left(\lambda_p/(p-1)\right)t}v(x,t)$ converges to an extremal of a Poincar\'e inequality on $\Omega$ with optimal constant $\lambda_p$, as $t\rightarrow \infty$. We also prove that the large time values of solutions approximate the extremals of a corresponding ``dual" Poincar\'e inequality on $\Omega$.  Moreover, our theory allows us to deduce the large time asymptotics of related doubly nonlinear flows involving various boundary conditions and nonlocal operators.
\end{abstract}

\section{Introduction}
This note concerns solutions $v:\Omega\times (0,\infty)\rightarrow \R$ of {\it Trudinger's equation}
\begin{equation}\label{TruEqn}
\partial_t(|v|^{p-2}v)=\Delta_pv.
\end{equation}
Here $p\in (1,\infty)$, $\Omega\subset\R^n$ is a bounded domain, and $\Delta_p$ is the $p-$Laplacian
\begin{equation}\label{pLapClassical}
\Delta_p\psi:=\text{div}(|D\psi|^{p-2}D\psi).
\end{equation}
This equation reduces to the standard heat equation when $p=2$, so it is a type of nonlinear diffusion. Trudinger's equation is known in the literature as a doubly nonlinear evolution, and it is distinctive among doubly 
nonlinear evolutions as it is homogeneous. For if $v$ is a solution, then any multiple of $v$ is also a solution. 

\par The PDE \eqref{TruEqn} is a special case of a general class of parabolic equations originally considered by Trudinger \cite{Tru}. He was able to generalize previous efforts of Moser \cite{Moser} and show that nonnegative solutions satisfy a Harnack inequality. Recently, Trudinger's result has been extended to nonnegative solutions which satisfy \eqref{TruEqn} weakly with respect to a certain doubling measure \cite{Kuusi}.  We also remark that positive viscosity solutions were recently shown to exist for $p\ge 2$ in \cite{BhaMar}.
\newline

\par {\bf Dirichlet boundary condition}. Our goal is to infer the large time behavior of solutions of \eqref{TruEqn}. The prototypical initial value problem we will focus on is
\begin{equation}\label{TruIVP}
\begin{cases}
\partial_t(|v|^{p-2}v)=\Delta_pv\quad &\text{in}\;\Omega\times (0,\infty)\\
\hspace{.67in}v=0 \quad &\text{on}\;\partial\Omega\times[0,\infty)\\
\hspace{.67in}v=g\quad &\text{on}\;\Omega\times\{0\}.
\end{cases}
\end{equation}
Here $g:\Omega\rightarrow \R$ is a given initial value function.  In what follows, we will learn much about the large time behavior of solutions by carefully studying their compactness and monotonicity properties and by taking advantage of the homogeneity of Trudinger's equation.

\par A key monotonicity feature of solutions of \eqref{TruIVP} is
\begin{equation}\label{StrongQuotientMon}
\frac{d}{dt}\left\{\frac{\displaystyle\int_\Omega|Dv(x,t)|^pdx}{\displaystyle\int_\Omega|v(x,t)|^pdx}\right\}\le 0.
\end{equation}
Therefore, we expect the flow \eqref{TruIVP} to be related to the {\it Poincar\'e inequality}
\begin{equation}\label{Poincare}
\lambda_p\int_{\Omega}|u|^pdx\le \int_\Omega |Du|^pdx\quad (u\in W^{1,p}_0(\Omega)).
\end{equation}
Here $\lambda_p$ is the largest constant $c$ such that $c\int_{\Omega}|u|^pdx\le \int_\Omega |Du|^pdx$ holds for each $u\in W^{1,p}_0(\Omega)$; so in this sense, $\lambda_p$ is optimal. Extremal functions are those for which equality holds in \eqref{Poincare}.  Recall that a function $u$ is extremal for \eqref{Poincare} if and only if $u$ satisfies the PDE 
\begin{equation}\label{peeEigenvalue}
\begin{cases}
-\Delta_pu=\lambda_p|u|^{p-2}u\quad&\text{in}\;\Omega\\
\hspace{.34in}u=0\quad&\text{on}\;\partial\Omega.
\end{cases}
\end{equation}

\par Another monotonicity property of solutions to \eqref{TruIVP} that will be even more important for us is this
\begin{equation}\label{WeakQuotientMonID}
\frac{d}{dt}\left\{\frac{\displaystyle\||v(\cdot,t)|^{p-2}v(\cdot,t)\|^q_{L^{q}(\Omega)}}{\displaystyle\||v(\cdot,t)|^{p-2}v(\cdot,t)\|^q_{W^{-1,q}(\Omega)}}\right\}\le 0.
\end{equation}
Here and throughout $q:=\frac{p}{p-1}$ is the H\"older exponent conjugate to $p$. This monotonicity suggests that the initial value problem \eqref{TruIVP} improves how $|v(\cdot,t)|^{p-2}v(\cdot,t)$ satisfies the inequality 
\begin{equation}\label{dualPoincare}
\mu_p\| f\|_{W^{-1,q}(\Omega)}^q\le \int_\Omega|f|^qdx\quad (f\in L^{q}(\Omega))
\end{equation}
as $t$ increases. Here $\mu_p:=\lambda_p^{\frac{1}{p-1}}$. We call this inequality the {\it dual Poincar\'e inequality} as equality holds if and only if $f=|u|^{p-2}u$ where $u$ is extremal for the Poincar\'e inequality \eqref{Poincare} (Appendix \ref{AppDPI}).

\par Our main result regarding \eqref{TruIVP} is as follows.  We postpone the definition of a weak solution to \eqref{TruIVP} until the following section. 

\begin{customthm}{1}\label{TruThm1} (i) Assume $v$ is a weak solution of \eqref{TruIVP}.  Then the limit 
\begin{equation}\label{scaledLim}
u:=\lim_{t\rightarrow\infty}e^{\left(\frac{\lambda_p}{p-1}\right)t}v(\cdot,t)
\end{equation}
exists in $L^p(\Omega)$ and $u$ is extremal for \eqref{Poincare}. If $u\not\equiv 0$, then
$v(\cdot,t)\not\equiv 0$ for all $t\ge 0$ and 
$$
\mu_p=\lim_{t\rightarrow\infty}\frac{\||v(\cdot,t)|^{p-2}v(\cdot,t)\|^q_{L^q(\Omega)}}{\||v(\cdot,t)|^{p-2}v(\cdot,t)\|^q_{W^{-1,q}(\Omega)}}.
$$ 
(ii) There is a weak solution $v$ of \eqref{TruIVP} such that the limit \eqref{scaledLim}
exists in $W_0^{1,p}(\Omega)$.  If $u\not\equiv 0$, 
$$
\lambda_p=\lim_{t\rightarrow\infty}\frac{\displaystyle\int_\Omega|Dv(x,t)|^pdx}{\displaystyle\int_\Omega|v(x,t)|^pdx}.
$$ 
\end{customthm}

\par We remark that it is possible that the limit \eqref{scaledLim} vanishes identically. However, we will show that for certain initial conditions this degeneracy does not occur. See Proposition \ref{NondegeneracyProp} below for more on this technical point.  We also emphasize that our methods are not restricted to solutions which are nonnegative.  The large time behavior of nonnegative solutions of related doubly nonlinear evolutions have been studied in various contexts including  \cite{Agueh, ManVes, SavVes, StaVaz}. In particular, in reference \cite{StaVaz}, the uniform convergence of $e^{\left(\lambda_p/(p-1)\right)t}v(\cdot,t)$ is verified for nonnegative solutions of \eqref{TruIVP} provided $\partial\Omega$ is $C^{2,\alpha}$. In proving Theorem \ref{TruThm1}, we will not make any regularity assumptions on $\partial\Omega$. 
\newline

\par {\bf Robin boundary condition}. Next we will consider the large time behavior of weak solutions of the flow
\begin{equation}\label{TruRBC}
\begin{cases}
\hspace{1in}\partial_t(|v|^{p-2}v)=\Delta_pv &\text{in}\;\Omega\times (0,\infty)\\
|Dv|^{p-2}Dv\cdot \nu +\beta|v|^{p-2}v =0 &\text{on}\;\partial\Omega\times[0,\infty)\\
\hspace{1.65in}v=g&\text{on}\;\Omega\times\{0\}.
\end{cases}
\end{equation}
We will assume that $\beta>0$ and that $\partial\Omega$ is $C^1$ with outward unit normal $\nu$.  The optimal Poincar\'e inequality 
\begin{equation}\label{PoincareRBC}
\lambda_p\int_\Omega|u|^pdx\le \int_\Omega|Du|^pdx+\beta \int_{\partial\Omega}|Tu|^pd\sigma \quad (u\in W^{1,p}(\Omega))
\end{equation}
\cite{Bucur,DaiFu} and its dual will play an important role in our analysis. Here  $T: W^{1,p}(\Omega)\rightarrow L^p(\partial\Omega; \sigma)$ is the Sobolev Trace operator and $\sigma$ is $n-1$ dimensional Hausdorff measure.  Adapting the methods used to prove Theorem \ref{TruThm1}, we will characterize the large time behavior of weak solutions of \eqref{TruRBC} in Theorem \ref{TruThmRBC} below. 
\newline

\par {\bf Neumann boundary condition}.
Then we will consider solutions of Trudinger's equation \eqref{TruEqn} which satisfy a Neumann boundary condition 
\begin{equation}\label{TruNBC}
\begin{cases}
\hspace{.22in}\partial_t(|v|^{p-2}v)=\Delta_pv &\text{in}\;\Omega\times (0,\infty)\\
|Dv|^{p-2}Dv\cdot \nu =0 &\text{on}\;\partial\Omega\times[0,\infty)\\
\hspace{.88in}v=g &\text{on}\;\Omega\times\{0\}.
\end{cases}
\end{equation}
Again we will assume that $\partial\Omega$ is $C^1$ with outward unit normal $\nu$.  For this initial value problem, we will employ the following optimal Poincar\'e inequality: for each $u\in W^{1,p}(\Omega)$ that satisfies 
\begin{equation}\label{pAverageZero}
\int_\Omega|u|^{p-2}u\;dx=0,
\end{equation}
we have
\begin{equation}\label{PoincareNBC}
\lambda_p\int_\Omega|u|^pdx\le \int_\Omega|Du|^pdx
\end{equation}
(see \cite{EspNit} and the references therein). We will show that if the ratio of any two nonvanishing extremal functions of \eqref{PoincareNBC} is constant, then a characterization of the large time behavior of solutions to \eqref{TruNBC} as in Theorem \ref{TruThm1} holds.  
This assertion is detailed in Theorem \ref{TruThmNBC} below. \newline

\par {\bf Fractional Trudinger equation}.
Finally, we will study an initial value problem involving a fractional version of Trudinger's equation 
\begin{equation}\label{TruFrac}
\begin{cases}
\partial_t(|v|^{p-2}v)+(-\Delta_p)^sv=0 &\text{in}\;\Omega\times (0,\infty)\\
\hspace{1.45in}v =0 &\text{on}\;(\R^n\setminus\Omega)\times[0,\infty)\\
\hspace{1.45in}v =g &\text{on}\;\Omega\times\{0\}.
\end{cases}
\end{equation}
Here $s\in (0,1)$, and $(-\Delta_p)^s$ is the fractional $p-$Laplacian
\begin{equation}\label{DeltaSpSmooth}
(-\Delta_p)^s\psi(x):=2\lim_{\epsilon\rightarrow 0^+}\int_{\R^n\setminus B_\epsilon(x)}\frac{|\psi(x)-\psi(y)|^{p-2}(\psi(x)-\psi(y))}{|x-y|^{n+ps}}dy \quad (x\in \R^n).
\end{equation}
The Poincar\'e inequality most naturally associated with this flow is 

\begin{equation}\label{FracPoincare}
\lambda_p\int_\Omega|u|^pdx\le \iint_{\R^n\times\R^n}\frac{|u(x)-u(y)|^p}{|x-y|^{n+ps}}dxdy\quad (u\in W^{s,p}_0(\Omega)).
\end{equation}
As before, $\lambda_p$ is chosen to be optimal, and the dual of \eqref{FracPoincare} will be central to our analysis. We also refer interested readers to \cite{Hitch} for various features of the fractional Sobolev space $W^{s,p}_0(\Omega)$.  In Theorem \ref{TruThmFrac} below, we will prove a large time limit for appropriately scaled solutions of \eqref{TruFrac} that is analogous to Theorem \ref{TruThm1}.
\newline

\par This paper is organized as follows. In Section \ref{DBC}, we shall introduce weak solutions of \eqref{TruIVP} and establish various monotonicity and compactness properties of solutions. Next, we will use these results to prove Theorem \ref{TruThm1} in Section \ref{LargeTimeSect}.  In Sections \ref{RBC}, \ref{NBC} and \ref{NonlocalSect}, we will verify analogs of Theorem \ref{TruThm1} for the flows \eqref{TruRBC}, \eqref{TruNBC} and \eqref{TruFrac}, respectively. 
Much of this work was completed in MIT's Norbert Wiener common room, the authors wish to express their gratitude to the MIT mathematics department for its warm hospitality. The authors are also appreciative of the insights Matteo Bonforte provided on a preliminary version of this work.

\section{Weak solutions}\label{DBC}
A natural identity associated with smooth solutions of \eqref{TruIVP} is  
\begin{equation}\label{LpMon}
\frac{d}{dt}\int_{\Omega}\frac{1}{p}|v(x,t)|^pdx=-\frac{1}{p-1}\int_{\Omega}|Dv(x,t)|^pdx.
\end{equation} 
This identity follows from direct computation or from multiplying Trudinger's equation \eqref{TruEqn} by $v$ and integrating by parts. Integrating this identity in time gives
$$
\frac{1}{p}\int_{\Omega}|v(x,t)|^pdx + \frac{1}{p-1}\int^t_0\int_{\Omega}|Dv(x,s)|^pdxds=\frac{1}{p}\int_{\Omega}|g(x)|^pdx
$$
for all $t\ge 0$. These observations motivate the following definition of a weak  solution of \eqref{TruIVP}.

\begin{defn}
Assume $g\in L^p(\Omega)$. A {\it weak solution} of \eqref{TruIVP} is a function $v:\Omega\times[0,\infty)\rightarrow \R$ that satisfies:
$(i)$
\begin{equation}\label{NaturalBounds}
v\in L^\infty([0,\infty);L^p(\Omega))\cap L^p([0,\infty); W^{1,p}_0(\Omega));
\end{equation}
$(ii)$
\begin{equation}\label{TruWeakSoln}
\int^\infty_0\int_\Omega |v|^{p-2}v\psi_tdxdt=\int^\infty_0\int_\Omega |Dv|^{p-2}Dv\cdot D\psi dxdt
\end{equation}
for each $\psi\in C^\infty_c(\Omega\times(0,\infty))$; and $(iii)$
\begin{equation}\label{InitialG}
v(\cdot,0)=g.
\end{equation}
\end{defn}

\par In order to better interpret weak solutions, we will use a definition of the $p$-Laplacian more general than its classical expression \eqref{pLapClassical}. We now consider $-\Delta_p$ as a mapping  
$$
-\Delta_p: W^{1,p}_0(\Omega)\rightarrow W^{-1,q}(\Omega)
$$
such that for each $u\in W^{1,p}_0(\Omega)$
\begin{equation}\label{W1pzeroPlaplace}
\langle -\Delta_pu, \phi\rangle := \int_\Omega |Du|^{p-2}Du\cdot D\phi \;dx,\quad \phi\in W^{1,p}_0(\Omega).
\end{equation}
We leave it to the reader to check that $-\Delta_p$ is a bijection, and for each $u\in W^{1,p}_0(\Omega)$,
\begin{equation}\label{pLapIsometry}
\|-\Delta_pu\|_{W^{-1,q}(\Omega)}=\| u \|^{p-1}_{W^{1,p}_0(\Omega)}.
\end{equation}
Here and below, $\|u\|_{W^{1,p}_0(\Omega)}:=\left(\int_\Omega|Du|^pdx\right)^{1/p}$.

\par Instead of \eqref{TruWeakSoln}, it is equivalent to require that 
$$
\frac{d}{dt}\int_\Omega |v(x,t)|^{p-2}v(x,t)\phi(x)dx+\int_\Omega |Dv(x,t)|^{p-2}Dv(x,t)\cdot D\phi(x) dx=0
$$
holds in the sense of distributions on $(0,\infty)$ for each $\phi\in W^{1,p}_0(\Omega)$.  This is another way of expressing 
\begin{equation}\label{DNEDBC}
\partial_t(|v(\cdot,t)|^{p-2}v(\cdot,t))+(-\Delta_pv(\cdot,t))=0
\end{equation}
in $W^{-1,q}(\Omega)$ for almost every $t\in [0,\infty)$ (Chapter 3, Lemma 1.1 of \cite{Temam}). Therefore
\begin{equation}\label{DerivativeIdentity}
\|\partial_t(|v(\cdot,t)|^{p-2}v(\cdot,t))\|_{W^{-1,q}(\Omega)}=\| v(\cdot,t)\|^{p-1}_{W^{1,p}_0(\Omega)}\quad (\text{a.e.}\; t>0).
\end{equation}
It follows that $\partial_t(|v|^{p-2}v)\in L^q([0,\infty);W^{-1,q}(\Omega))$ and $t\mapsto |v(\cdot,t)|^{p-2}v(\cdot,t)$ is locally absolutely continuous with values in $W^{-1,q}(\Omega)$. Using the notation of Chapter 1 of \cite{AGS}, we have shown
\begin{equation}\label{ACqSpace}
|v|^{p-2}v\in AC^q_{\text{loc}}([0,\infty); W^{-1,q}(\Omega)).
\end{equation}

\par This continuity ensures $v$ is defined at each $t\ge 0$, and in particular, at time $0$ in \eqref{InitialG}. We also can use these observations to establish
that \eqref{LpMon} holds for every weak solution. 

\begin{lem}\label{BasicMonLemma}
Assume $v$ is a weak solution of \eqref{TruIVP}. Then $[0,\infty)\ni t\mapsto \int_\Omega|v(x,t)|^pdx$ is locally absolutely continuous and 
\eqref{LpMon} holds for almost every $t>0$.
\end{lem} 
\begin{proof} For $w\in W^{-1,q}(\Omega)$, define
$$
\Phi(w):=
\begin{cases}
\frac{1}{q}\int_\Omega|w|^qdx, &\quad  w\in L^q(\Omega)\\
+\infty, &\quad \text{otherwise}
\end{cases}.
$$
Note that $\Phi$ is convex, proper and lower semicontinuous. It is straightforward to verify 
$$
\partial \Phi(w):=\left\{\xi\in W^{1,p}_0(\Omega): \Phi(z)\ge \Phi(w)+\langle \xi, z-w\rangle\; \text{all $z$}\in W^{-1,q}(\Omega)\right\}
$$
is nonempty and equal to the singleton $\{|w|^{q-2}w\}$ if and only if 
$|w|^{q-2}w\in W^{1,p}_0(\Omega)$.   In this case, we write $|\partial \Phi|(w):=\||w|^{q-2}w\|_{W^{1,p}_0(\Omega)}$.

\par Observe
$$
|\partial\Phi|(|v|^{p-2}v)\cdot\|\partial_t(|v|^{p-2}v)\|_{W^{-1,q}(\Omega)}=\|v\|^p_{W^{1,p}_0(\Omega)}\in 
L^1_{\text{loc}}[0,\infty).
$$
Here we have used \eqref{DerivativeIdentity}. In view of \eqref{ACqSpace}, we also have that
$$
\Phi(|v(\cdot,t)|^{p-2}v(\cdot,t))=\frac{1}{q}\int_\Omega|v(x,t)|^pdx
$$
is a locally absolutely continuous function $[0,\infty)$ (see Remark 1.4.6 in \cite{AGS}, Proposition 4.11 in \cite{Vis}, or Lemma 4.1 in \cite{Colli}).  Moreover,
\begin{align*}
\frac{d}{dt}\frac{1}{q}\int_\Omega|v(x,t)|^pdx&=\frac{d}{dt}\Phi(|v(\cdot,t)|^{p-2}v(\cdot,t))\\
&=\langle \partial_t(|v(\cdot,t)|^{p-2}v(\cdot,t)),  \partial\Phi(|v(\cdot,t)|^{p-2}v(\cdot,t))\rangle \\
&=\langle \partial_t(|v(\cdot,t)|^{p-2}v(\cdot,t)), v(\cdot,t)\rangle \\
&=-\langle -\Delta_pv(\cdot,t), v(\cdot,t)\rangle \\
&=- \|v(\cdot,t)\|^p_{W^{1,p}_0(\Omega)}\\
&=- \int_\Omega|Dv(x,t)|^pdx
\end{align*}
for almost every $t>0$.
\end{proof}

The first indication of how the scaling mentioned in Theorem \ref{TruThm1} arises can be seen in the following corollary. 
\begin{cor}\label{MonEthenLp}
Assume $v$ is a weak solution of \eqref{TruIVP}. Then 
$$
\frac{d}{dt}\left\{e^{p\left(\frac{\lambda_p}{p-1}\right)t}\int_\Omega|v(x,t)|^pdx\right\}\le 0
$$
for almost every $t>0$.
\end{cor}
\begin{proof}
By Lemma \ref{BasicMonLemma} and the Poincar\'e inequality \eqref{Poincare}, 
$$
\frac{d}{dt}\int_\Omega|v(x,t)|^pdx=-\frac{p}{p-1}\int_\Omega|Dv(x,t)|^pdx\le -p\left(\frac{\lambda_p}{p-1}\right)\int_\Omega|v(x,t)|^pdx.
$$
The assertion now follows by the product rule. 
\end{proof}

\begin{cor}\label{UnifContLp}
Assume $v$ is a weak solution of \eqref{TruIVP}.  Then $v: [0,\infty)\rightarrow L^p(\Omega)$ is bounded and uniformly continuous. 
\end{cor}
\begin{proof} We will first establish the continuity of $v: [0,\infty)\rightarrow L^p(\Omega)$.  Let $t_k\in [0,\infty)$ and suppose $t_k\rightarrow t$. By \eqref{ACqSpace},
$$ 
|v(\cdot,t_k)|^{p-2}v(\cdot,t_k)\rightarrow |v(\cdot,t)|^{p-2}v(\cdot,t)\;\;\text{in}\;\; W^{-1,q}(\Omega).
$$
In view of \eqref{NaturalBounds}, $(|v(\cdot,t_k)|^{p-2}v(\cdot,t_k))_{k\in \N}\subset L^q(\Omega)$ is also bounded.  It then follows from a routine weak convergence argument that  
\begin{equation}\label{vDoesIndeedConvWeakCo24}
|v(\cdot,t_k)|^{p-2}v(\cdot,t_k)\rightharpoonup |v(\cdot,t)|^{p-2}v(\cdot,t)\;\;\text{in}\;\; L^{q}(\Omega).
\end{equation}
\par We can now invoke Lemma \ref{BasicMonLemma} in order to deduce
\begin{align*}
\lim_{k\rightarrow\infty}\| |v(\cdot,t_k)|^{p-2}v(\cdot,t_k)\|^q_{L^q(\Omega)}&=\lim_{k\rightarrow\infty}\int_\Omega|v(x,t_k)|^pdx\\
&=\int_\Omega|v(x,t)|^pdx\\
&=\| |v(\cdot,t)|^{p-2}v(\cdot,t)\|^q_{L^q(\Omega)}.
\end{align*}
Combining this convergence with \eqref{vDoesIndeedConvWeakCo24} gives  
$$
|v(\cdot,t_k)|^{p-2}v(\cdot,t_k)\rightarrow |v(\cdot,t)|^{p-2}v(\cdot,t)\;\; \text{in}\;\; L^{q}(\Omega) 
$$
(Chapter 1, Theorem 1 of \cite{EG}). As a result, $v(\cdot,t_k)\rightarrow v(\cdot,t)$ in $L^{p}(\Omega)$. We conclude that $v: [0,\infty)\rightarrow L^p(\Omega)$ is continuous, as claimed. 

\par The previous corollary implies  
$$
\int_\Omega|v(x,t)|^pdx\le e^{-p\left(\frac{\lambda_p}{p-1}\right)t} \int_\Omega|g(x)|^pdx
$$
for all $t\ge 0$. Consequently, $v: [0,\infty)\rightarrow L^p(\Omega)$ is bounded and tends to $0$, as $t\rightarrow \infty$. It also follows that this function is necessarily 
uniformly continuous. 
\end{proof}
Next we will establish \eqref{WeakQuotientMonID}, which is an important monotonicity formula for weak solutions in relation to the dual Poincar\'e inequality 
\eqref{dualPoincare}. This observation was inspired by our previous study on curves of maximal slope \cite{HynLin}.  

\begin{prop}\label{WeakQuotientMon}
Assume $v$ is a weak solution of \eqref{TruIVP} with $v(\cdot,t)\not\equiv 0$ for $t\ge 0$.  Then \eqref{WeakQuotientMonID} holds for almost every $t>0$. In particular, 
$$
[0,\infty)\ni t\mapsto \frac{\||v(\cdot,t)|^{p-2}v(\cdot,t)\|^q_{L^{q}(\Omega)}}{\displaystyle\||v(\cdot,t)|^{p-2}v(\cdot,t)\|^q_{W^{-1,q}(\Omega)}}
$$
is nonincreasing. 
\end{prop}
\begin{proof}
Set $w:=|v|^{p-2}v$.  By \eqref{ACqSpace}, $w\in AC^q_{\text{loc}}([0,\infty); W^{-1,q}(\Omega))$. As $v$ is a  solution of \eqref{DNEDBC},  $w$
satisfies 
$$
(-\Delta_p)^{-1}(\partial_tw(\cdot,t))+|w(\cdot,t)|^{q-2}w(\cdot,t)=0
$$
for almost every $t>0$. Below, we will perform several computations involving $w$ where we suppress the time dependence of $w$ for notational ease. 

\par First, we compute 
\begin{equation}\label{Wmon1}
\frac{d}{dt}\frac{1}{q}\|w\|^q_{L^q(\Omega)}=-\|\partial_tw\|^q_{W^{-1,q}(\Omega)}.
\end{equation}
This computation can be performed exactly as we did for \eqref{LpMon} in the proof of Lemma \ref{BasicMonLemma}. Next, we have the 
general formula
\begin{equation}\label{WAbsForm}
\frac{d}{dt}\frac{1}{q}\|w\|^q_{W^{-1,q}(\Omega)}=\langle \partial_tw,(-\Delta_p)^{-1}w \rangle,
\end{equation}
which is valid for any $w\in AC^q_{\text{loc}}([0,\infty); W^{-1,q}(\Omega))$ (see Remark 1.4.6 in \cite{AGS}).  

\par Also observe that by \eqref{DerivativeIdentity},
\begin{align*}
\left|\langle \partial_tw,(-\Delta_p)^{-1}w \rangle\right|&\le \|\partial_tw\|_{W^{-1,q}(\Omega)}
\|(-\Delta_p)^{-1}w \|_{W^{1,p}_0(\Omega)}\\
&=\|\partial_tw\|_{W^{-1,q}(\Omega)}\|w \|^{q-1}_{W^{-1,q}(\Omega)}
\end{align*}
and 
\begin{align*}
\|w\|^q_{L^q(\Omega)}&=\langle w, |w|^{q-2}w \rangle\\
&=-\langle w, (-\Delta_p)^{-1}(\partial_tw)\rangle\\
&\le\|w \|_{W^{-1,q}(\Omega)} \|(-\Delta_p)^{-1}(\partial_tw) \|_{W^{1,p}_0(\Omega)}\\
&=\|w \|_{W^{-1,q}(\Omega)} \|\partial_tw \|^{q-1}_{W^{-1,q}(\Omega)}.
\end{align*}
Therefore, 
\begin{equation}\label{WproductEstimate}
\|w\|^q_{L^q(\Omega)}\left|\langle \partial_tw,(-\Delta_p)^{-1}w \rangle\right|\le \|w \|^q_{W^{-1,q}(\Omega)} \|\partial_tw \|^{q}_{W^{-1,q}(\Omega)}.
\end{equation}
\par Combining \eqref{Wmon1}, \eqref{WAbsForm} and \eqref{WproductEstimate} give
\begin{align*}
\frac{d}{dt}\frac{\|w\|^q_{L^q(\Omega)}}{\|w\|^q_{W^{-1,q}(\Omega)}}&=
\frac{-q\|\partial_tw\|^{q}_{W^{-1,q}(\Omega)}\|w\|^{q}_{W^{-1,q}(\Omega)}
-q\|w\|^q_{L^q(\Omega)}\langle \partial_tw,(-\Delta_p)^{-1}w \rangle
}{\|w\|^{2q}_{W^{-1,q}(\Omega)}}\\
&=\frac{-q}{\|w\|^{2q}_{W^{-1,q}(\Omega)}}\left(\|\partial_tw\|^{q}_{W^{-1,q}(\Omega)}\|w\|^{q}_{W^{-1,q}(\Omega)}-\|w\|^q_{L^q(\Omega)}\langle -\partial_tw,(-\Delta_p)^{-1}w\rangle \right)\\
&\le 0
\end{align*}
for almost every $t>0$.
\end{proof}
\begin{rem}\label{StrongMonRemark}
We do not know if the monotonicity \eqref{StrongQuotientMon} holds for every weak solution of \eqref{TruIVP}. However, it is not hard to show it holds for each 
smooth solution $v$ that is nonvanishing.  By integrating by parts and using H\"older's inequality, we have
\begin{align*}
\frac{1}{p-1}\int_{\Omega}|Dv|^pdx &=-\frac{1}{p-1}\int_{\Omega}\Delta_pv\cdot vdx  \\
&=-\int_{\Omega}|v|^{p-2}v v_tdx \\
&=-\int_{\Omega}(|v|^{p/2-2}v v_t)|v|^{p/2}dx \\
&\le \left(\int_\Omega |v|^pdx\right)^{1/2}\left(\int_\Omega |v|^{p-2}|v_t|^2dx\right)^{1/2}.
\end{align*}
Direct computation then gives 
$$
\frac{d}{dt}\frac{\int_{\Omega}|Dv|^pdx}{\int_{\Omega}|v|^pdx}=\frac{p(p-1)}{\left(\int_{\Omega}|v|^pdx\right)^2}\left\{\left(\frac{1}{p-1}\int_{\Omega}|Dv|^pdx\right)^2-\int_{\Omega}|v|^{p-2}|v_t|^2dx \int_{\Omega}|v|^{p}dx\right\},
$$
which implies \eqref{StrongQuotientMon}. See Theorem 3.1 of \cite{ManVes} and Lemma 2.1 of \cite{SavVes} for similar computations.
\end{rem}

Note that if the initial condition $g$ is extremal for the Poincar\'e inequality \eqref{Poincare}, then 
\begin{equation}\label{SepofVar}
v(x,t)=e^{-\left(\frac{\lambda_p}{p-1}\right)t}g(x)
\end{equation}
is a solution of \eqref{TruIVP}. Theorem \ref{TruThm1} asserts all weak solutions exhibit this separation of variables type behavior as $t$ tends to $\infty$. 
In view of the monotonicity of the previous proposition, we will show that the expression above is the only weak solution of \eqref{TruIVP} with initial condition $g$. 

\begin{cor}\label{OnlyUniqueness}
Assume $v$ is a weak solution of \eqref{TruIVP} and that $v(\cdot,0)=g$ is extremal for the Poincar\'e inequality \eqref{Poincare}. Then $v$ is necessarily given by 
\eqref{SepofVar}.
\end{cor}
\begin{proof}
Suppose $g\equiv 0$, then $\int_\Omega|v(x,t)|^pdx \le \int_\Omega|g(x)|^pdx=0$ for all $t\ge 0$. Hence, \eqref{SepofVar} holds. 

\par Now assume $g\not\equiv 0$. By \eqref{ACqSpace}, there is $T>0$ for which $v(\cdot,t)\not\equiv 0$ for $t\in [0,T)$. On this interval, 
$$
\frac{\||v(\cdot,t)|^{p-2}v(\cdot,t)\|^q_{L^{q}(\Omega)}}{\displaystyle\||v(\cdot,t)|^{p-2}v(\cdot,t)\|^q_{W^{-1,q}(\Omega)}}
\le \frac{\||g|^{p-2}g\|^q_{L^{q}(\Omega)}}{\displaystyle\||g|^{p-2}g\|^q_{W^{-1,q}(\Omega)}}=\mu_p.
$$
As a result, $|v(\cdot,t)|^{p-2}v(\cdot,t)$ is extremal for the dual Poincar\'e inequality \eqref{dualPoincare}, and $v(\cdot,t)$ is extremal for the 
Poincar\'e inequality \eqref{Poincare} for each $t\in [0,T)$. In view of \eqref{DNEDBC} and \eqref{peeEigenvalue}, we have
$$
\partial_t(|v(\cdot,t)|^{p-2}v(\cdot,t))=\Delta_pv(\cdot,t)=-\lambda_p|v(\cdot,t)|^{p-2}v(\cdot,t).
$$
Therefore, $|v(\cdot,t)|^{p-2}v(\cdot,t)=e^{-\lambda_pt}|g|^{p-2}g$ which gives \eqref{SepofVar} on $[0,T)$. 

\par Finally, observe that this argument actually implies we could have chosen $T=+\infty$ from the outset. For if $T$ is the first time that $v(\cdot,T)=e^{-\left(\lambda_p/(p-1)\right)T}g\equiv 0$, then $g$ would have to vanish identically. So if $g\not\equiv 0$, then it must be that $v$ does not vanish identically on $[0,\infty)$.
\end{proof}
Let us now discuss compactness properties of weak solutions. 
\begin{prop}\label{CompactTrudinger}
Assume $(v^k)_{k\in \N}$ is a sequence of weak solutions of \eqref{TruIVP} with $v^k(\cdot,0)=g^k$ and 
$$
\sup_{k\in\N}\int_\Omega|g^k|^pdx<\infty.
$$
There is a subsequence $(v^{k_j})_{j\in \N}$ and $v$ satisfying \eqref{NaturalBounds} such that 
\begin{equation}\label{StrongLocW1p}
v^{k_j}\rightarrow v\;\; \text{in}\;\; L^p_{\text{loc}}([0,\infty); W^{1,p}_0(\Omega)),
\end{equation}
\begin{equation}\label{LocUnifComLp}
v^{k_j}\rightarrow v\;\; \text{in}\;\; C_{\text{loc}}([0,\infty); L^p(\Omega)),
\end{equation}
and
\begin{equation}\label{StrongTimeDerConv}
\partial_t(|v^{k_j}|^{p-2}v^{k_j})\rightarrow \partial_t(|v|^{p-2}v)\;\; \text{in}\;\; L^q_{\text{loc}}([0,\infty); W^{-1,q}(\Omega)).
\end{equation}
Moreover, $v$ is a weak solution of \eqref{TruIVP} with $v(\cdot,0)=g$, where $|g|^{p-2}g$ is a weak limit of $(|g^{k_j}|^{p-2}g^{k_j})_{j\in \N}$ in $L^q(\Omega)$.
\end{prop}

\begin{proof}
1. Set $C:=\sup_{k\in\N}\int_\Omega|g^k|^pdx$. By Lemma \ref{BasicMonLemma},
\begin{equation}\label{kkNaturalBounds}
\sup_{t\ge 0}\int_\Omega|v^k(x,t)|^pdx+\int^\infty_0\int_\Omega|Dv^k(x,t)|^pdx \le C
\end{equation}
for each $k\in \N$.  
It then follows from \eqref{DerivativeIdentity} that
\begin{equation}\label{dualkkNaturalBounds}
\sup_{t\ge 0}\| |v^k(\cdot,t)|^{p-2}v^k(\cdot,t)\|^q_{L^q(\Omega)}+\int^\infty_0\|\partial_t(|v^k(\cdot,t)|^{p-2}v^k(\cdot,t))\|^q_{W^{-1,q}(\Omega)}dt\le C
\end{equation}
uniformly in $k\in \N$. 

\par As $q>1$ and $L^q(\Omega)\subset W^{-1,q}(\Omega)$ with compact embedding, it follows that there is a subsequence $(v^{k_j})_{j\in \N}$ and 
$w: [0,\infty)\rightarrow W^{-1,q}(\Omega)$ such that 
$$
|v^{k_j}|^{p-2}v^{k_j}\rightarrow w\;\; \text{in}\;\; C_{\text{loc}}([0,\infty); W^{-1,q}(\Omega)) 
$$
as $j\rightarrow\infty$ \cite{Aubin, Simon}. By weak convergence, 
$$
\partial_t(|v^{k_j}|^{p-2}v^{k_j})\rightharpoonup \partial_tw\;\; \text{in}\;\; L^q([0,\infty); W^{-1,q}(\Omega)),
$$
as $j\rightarrow\infty$. Without any loss of generality, we may also assume there is $v$ for which
\begin{equation}\label{WeakConvvkayJay}
v^{k_j}\rightharpoonup  v\;\; \text{in}\;\; L^p([0,\infty); W_0^{1,p}(\Omega)),
\end{equation}
and that there is $\xi$ such that
$$
|Dv^{k_j}|^{p-2}Dv^{k_j}\rightharpoonup  \xi\;\; \text{in}\;\; L^q(\Omega\times [0,\infty); \R^n)
$$ 
as $j\rightarrow\infty$.  Moreover, 
\begin{equation}\label{VeryWEAKeqn}
\partial_t w=\text{div}(\xi)
\end{equation}
in the sense of distributions on $\Omega\times (0,\infty)$.

\par 2. For any interval $[t_0,t_1]\subset [0,\infty)$
$$
\int^{t_1}_{t_0}\int_\Omega|v^k|^pdxdt=\int^{t_1}_{t_0}\int_\Omega|v^k|^{p-2}v^k\cdot v^kdxdt=
\int^{t_1}_{t_0}\langle |v^k(\cdot,t)|^{p-2}v^k(\cdot,t) , v^k(\cdot,t)\rangle dt.
$$
As a result, 
\begin{equation}\label{StronginLp}
\lim_{j\rightarrow \infty}\int^{t_1}_{t_0}\int_\Omega|v^{k_j}|^pdxdt=\int^{t_1}_{t_0}\langle w(\cdot,t) , v(\cdot,t)\rangle dt.
\end{equation}
In particular, for $u\in L^p([0,\infty); W^{1,p}_0(\Omega) )$
\begin{align*}
0&\le \lim_{j\rightarrow \infty}\int^{t_1}_{t_0}\int_\Omega\left(|v^{k_j}|^{p-2}v^{k_j} - |u|^{p-2}u \right)(v^k- u)dxdt\\
&=\lim_{j\rightarrow \infty}\left[\int^{t_1}_{t_0}\langle|v^k(\cdot,t)|^{p-2}v^k(\cdot,t) , v^k(\cdot,t) -u(\cdot,t)\rangle dt  - \int^{t_1}_{t_0}\int_\Omega|u|^{p-2}u (v^k- u)dxdt\right]\\
&=\int^{t_1}_{t_0}\langle w(\cdot,t),v(\cdot,t)- u(\cdot,t)\rangle dt - \int^{t_1}_{t_0}\int_\Omega|u|^{p-2}u (v- u)dxdt   \\
&=\int^{t_1}_{t_0}\langle w(\cdot,t)-|u(\cdot,t)|^{p-2}u(\cdot,t),v(\cdot,t)- u(\cdot,t)\rangle dt.
\end{align*}
\par We can now choose $u=v-\tau \phi$ for $\phi\in C^\infty_c(\Omega\times[0,\infty))$ and $\tau>0$ to get 
$$
0\le \tau\int^{t_1}_{t_0}\langle w(\cdot,t)-|v(\cdot,t)-\tau\phi(\cdot,t)|^{p-2}(v(\cdot,t)-\tau\phi(\cdot,t)),\phi(\cdot,t)\rangle dt.
$$
Cancelling $\tau$ and then sending $\tau\rightarrow 0^+$ gives
$$
0\le \int^{t_1}_{t_0}\langle w(\cdot,t)-|v(\cdot,t)|^{p-2}v(\cdot,t),\phi(\cdot,t)\rangle dt.
$$
As a result
\begin{equation}\label{wisVtothePminus1}
w=|v|^{p-2}v.
\end{equation}

\par 3. Now it follows from \eqref{WeakConvvkayJay} and \eqref{StronginLp} that 
$$
v^{k_j}\rightarrow v\;\; \text{in}\;\; L^p_{\text{loc}}([0,\infty); L^{p}(\Omega)),
$$
as $j\rightarrow \infty$.  We may also assume  
\begin{equation}\label{PointwiseConv}
v^{k_j}(\cdot,t)\rightarrow v(\cdot,t)\;\; \text{in}\;\; L^{p}(\Omega)
\end{equation}
for almost every $t>0$, since this type of convergence happens for a subsequence of  $(v^{k_j})_{j\in \N}$. 
In view of \eqref{kkNaturalBounds}, we also see that $v$ satisfies \eqref{NaturalBounds}. Now let $t_0, t_1\in [0,\infty)$ $(t_0<t_1)$ be 
two such times for which \eqref{PointwiseConv} occurs.  By Lemma \ref{BasicMonLemma},
\begin{align}\label{StronginW1p}
\lim_{j\rightarrow \infty}\int^{t_1}_{t_0}\int_\Omega|Dv^{k_j}(x,t)|^pdxdt 
&= -\left(\frac{p-1}{p}\right)\lim_{j\rightarrow \infty}\left[\int_\Omega|v^{k_j}(x,t_1)|^pdx -\int_\Omega|v^{k_j}(x,t_0)|^pdx \right] \nonumber \\
&=-\left(\frac{p-1}{p}\right)\left[\int_\Omega|v(x,t_1)|^pdx -\int_\Omega|v(x,t_0)|^pdx \right].
\end{align}

\par From \eqref{VeryWEAKeqn} and \eqref{wisVtothePminus1}, 
\begin{equation}\label{VeryWEAKeqn2}
\partial_t(|v|^{p-2}v)=\text{div}(\xi)
\end{equation}
holds in the sense of distributions on $\Omega\times(0,\infty)$. The proof of Lemma \ref{BasicMonLemma} can be now adapted to show that $t\mapsto \int_\Omega|v(x,t)|^pdx$ is absolutely continuous
and satisfies 
$$
\frac{d}{dt} \frac{1}{p}\int_\Omega|v(x,t)|^pdx = -\frac{1}{p-1} \int_\Omega D\xi(x,t)\cdot Dv(x,t)dx
$$
for almost every $t\ge 0$. Combining with \eqref{StronginW1p}, we have 
\begin{equation}\label{StronginW1p2}
\lim_{j\rightarrow \infty}\int^{t_1}_{t_0}\int_\Omega|Dv^{k_j}(x,t)|^pdxdt=\int^{t_1}_{t_0}\int_\Omega \xi(x,t)\cdot Dv(x,t)dxdt.
\end{equation}
Now we can employ virtually the same argument used to verify \eqref{wisVtothePminus1} in order to deduce 
$$
\xi = |Dv|^{p-2}Dv.
$$
In view of \eqref{StronginW1p2}, we conclude \eqref{StrongLocW1p}; and by \eqref{VeryWEAKeqn2}, we see that $v$ is a weak solution of \eqref{TruIVP}. 

\par It also follows from \eqref{StrongLocW1p} that 
\begin{align*}
\lim_{j\rightarrow \infty}\int^{t_1}_{t_0}\|\partial_t(|v^{k_j}(\cdot,t)|^{p-2}v^{k_j}(\cdot,t))\|^q_{W^{-1,q}(\Omega)}dt&=\lim_{j\rightarrow \infty}\int^{t_1}_{t_0}\int_\Omega|Dv^{k_j}(x,t)|^pdxdt\\
&=\int^{t_1}_{t_0}\int_\Omega |Dv(x,t)|^pdx\\
&=\int^{t_1}_{t_0}\|\partial_t(|v(\cdot,t)|^{p-2}v(\cdot,t))\|^q_{W^{-1,q}(\Omega)}dt
\end{align*}
for each interval $[t_0,t_1]\subset [0,\infty)$. This verifies the assertion \eqref{StrongTimeDerConv}.

\par 4. Recall that for each $k$, the function $t\mapsto \int_\Omega|v^k(x,t)|^pdx$ is nonincreasing.  By Helly's Theorem (Lemma 3.3.3 in \cite{AGS}), we can pass to a further subsequence 
if necessary to find a nonincreasing function $f:[0,\infty)\rightarrow [0,\infty)$ such that 
$$
f(t)=\lim_{j\rightarrow \infty} \int_\Omega|v^{k_j}(x,t)|^pdx.
$$
for all $t\ge 0$. By the pointwise convergence \eqref{PointwiseConv}, we have 
\begin{equation}\label{fEqualIntegral}
f(t)=\int_\Omega|v(x,t)|^pdx\quad a.e.\; t>0.
\end{equation}
Recall the bound \eqref{dualkkNaturalBounds} and that the limit $|v(\cdot,t)|^{p-2}v(\cdot,t)=\lim_{j\rightarrow\infty}|v^{k_j}(\cdot,t)|^{p-2}v^{k_j}(\cdot,t)$ holds in $W^{-1,q}(\Omega)$. It follows that
$|v^{k_j}(\cdot,t)|^{p-2}v^{k_j}(\cdot,t)$ converges to $|v(\cdot,t)|^{p-2}v(\cdot,t)$ weakly in $L^q(\Omega)$ and
\begin{align}\label{flessThanIntegral}
f(t)&=\lim_{j\rightarrow \infty}\||v^{k_j}(\cdot,t)|^{p-2}v^{k_j}(\cdot,t)\|^q_{L^q(\Omega)}\nonumber \\
&\ge \||v(\cdot,t)|^{p-2}v(\cdot,t)\|^q_{L^q(\Omega)}\nonumber \\
&= \int_\Omega|v(x,t)|^pdx
\end{align}
for each $t\ge0$.

\par We claim that \eqref{fEqualIntegral} actually holds for every $t>0$. Fix $s>0$ and select $s_m\in (0,s)$ with $s=\lim_{m\rightarrow\infty}s_m$ such 
that \eqref{fEqualIntegral} holds for each $t=s_m$. Such a sequence exists as \eqref{fEqualIntegral} holds almost everywhere on $[0,\infty)$. Note that 
$f(s)\le f(s_m)$ for all $m\in \N$.  By \eqref{flessThanIntegral}
and Lemma \ref{BasicMonLemma}, 
\begin{align*}
f(s)&\le \liminf_{m\rightarrow \infty}f(s_m)\\
&= \liminf_{m\rightarrow \infty}\int_\Omega|v(x,s_m)|^pdx\\
&= \int_\Omega|v(x,s)|^pdx\\
&\le f(s).
\end{align*}
This computation verifies our claim that \eqref{fEqualIntegral} holds for every $t>0$ and in fact 
$$
\lim_{j\rightarrow \infty} \int_\Omega|v^{k_j}(x,t)|^pdx= \int_\Omega|v(x,t)|^pdx
$$
uniformly for $t$ belonging to compact subintervals of $(0,\infty)$ \cite{StackExchange}. A slight variation of the proof of 
Corollary \ref{UnifContLp} can now be used to verify \eqref{LocUnifComLp}.
\end{proof}
It has been established that there is a weak solution of \eqref{TruIVP} 
as defined above. See, for instance, any of the references \cite{AltLuck,Bernis, DiaThe, DiBShow, GraMin, Tsutsumi, Vis, Xu}.  
However, it is not known if weak solutions are unique unless they have better regularity than what is typically known for a general weak solution as explained in \cite{AltLuck, DiaThe}. 
Our purpose here is to show how that there is at least one weak solution of \eqref{TruIVP} which has some useful properties with regard to its large 
time behavior. To this end, we will study solutions of the following {\it implicit time scheme}: $(u^k)_{k\in \N}$, where 
\begin{equation}\label{ImplicitTimeScheme}
\begin{cases}
\displaystyle\frac{{\cal J}_p(u^k)- {\cal J}_p(u^{k-1})}{\tau}=\Delta_pu^k\quad &\text{in}\;\Omega \\
\hspace{1.15in}u^k=0\quad &\text{on}\;\partial\Omega
\end{cases}
\end{equation}
for each $k\in \N$ and $\tau>0$.  Here $u^0=g\in L^p(\Omega)$ and
$$
{\cal J}_p(z):=|z|^{p-2}z, \quad z\in \R.
$$

\par Standard methods from the calculus of variations can be used to show that there is a unique solution sequence 
$(u^k)_{k\in \N}\subset W^{1,p}_0(\Omega)$ of \eqref{ImplicitTimeScheme}. If we multiply the PDE in \eqref{ImplicitTimeScheme} by $u^k$ and integrate by parts, 
we find 
$$
\frac{1}{p}\int_\Omega|u^k|^pdx-\frac{1}{p}\int_\Omega|u^{k-1}|^pdx\le -\frac{\tau}{p-1}\int_\Omega|Du^{k}|^pdx\quad (k\in \N).
$$
This inequality is a discrete version of \eqref{LpMon} and implies 
\begin{equation}\label{DiscreteEnergyBound}
\sup_{k\in \N}\int_{\Omega}|u^k|^pdx + \tau \sum^\infty_{k=1}\int_\Omega|Du^k|^pdx \le \int_{\Omega}|g|^pdx.
\end{equation}
Alternatively, if we multiply the PDE in \eqref{ImplicitTimeScheme} by $u^k-u^{k-1}$ and integrate by parts, we get
\begin{equation}\label{DiscreteEnergyMon}
\frac{1}{p}\int_\Omega|Du^k|^pdx-\frac{1}{p}\int_\Omega|Du^{k-1}|^pdx\le -\int_\Omega\left(\frac{{\cal J}_p(u^k)- {\cal J}_p(u^{k-1})}{\tau}\right)(u^k-u^{k-1})dx\quad (k\in \N).
\end{equation}

\par For each $t\ge 0$, let us define
\begin{equation}\label{vTauAndwTau}
\begin{cases}
v_\tau(\cdot,t):=
\begin{cases}
g, \quad & t=0\\
u^k, \quad & t\in ((k-1)\tau,k\tau]
\end{cases}\\\\
w_\tau(\cdot,t):={\cal J}_p(u^{k-1})+\left(\frac{t- (k-1)\tau}{\tau}\right)\left({\cal J}_p(u^{k})-{\cal J}_p(u^{k-1})\right),\quad t\in [(k-1)\tau,k\tau].
\end{cases}
\end{equation}
It is evident that 
$$
\partial_t(w_\tau(\cdot,t))=\Delta_pv_\tau(\cdot,t)\quad \text{a.e.}\; t\ge 0.
$$
Employing \eqref{DiscreteEnergyBound} and \eqref{pLapIsometry}, it is routine to check 
$$
\sup_{t\ge 0}\int_\Omega|v_\tau(x,t)|^pdx + \int^\infty_0\int_\Omega|Dv_\tau(x,t)|^pdx dt\le \int_\Omega|g|^pdx 
$$
and
$$
\sup_{t\ge 0}\int_\Omega|w_\tau(x,t)|^qdx+\int^\infty_0\|\partial_t(w_\tau(\cdot,t))\|_{W^{-1,q}(\Omega)}^qdt\le  \int_\Omega|g|^pdx 
$$
for every $\tau>0$. And given that ${\cal J}_p$ is monotone, \eqref{DiscreteEnergyMon} implies  
\begin{equation}\label{preW1pMonForOneWS}
\int_\Omega|Dv_\tau(x,t)|^pdx\le \int_\Omega|Dv_\tau(x,s)|^pdx,\quad 0\le s\le t<\infty
\end{equation}
for each $\tau>0$.

\par Using ideas very similar to those used to prove Proposition \ref{CompactTrudinger}, we have the following assertion.  The main point for stating this assertion (in view of the known existence results) is the monotonicity \eqref{W1pMonForOneWS}, which is mainly due to \eqref{preW1pMonForOneWS}. 

\begin{prop}\label{ExistenceProp}
Let $g\in L^p(\Omega)$, and for each $\tau>0$, define $v_\tau$ and $w_\tau$ via \eqref{vTauAndwTau}. There is a sequence $(\tau_j)_{j\in \N}\subset(0,\infty)$ tending to $0$ and $v$ satisfying \eqref{NaturalBounds} such that 
$$
v_{\tau_j}\rightarrow v\;\; \text{in}\;\; L^p_{\text{loc}}([0,\infty); W^{1,p}_0(\Omega)),
$$
$$
v_{\tau_j}\rightarrow v\;\; \text{in}\;\; C_{\text{loc}}([0,\infty); L^p(\Omega)),
$$
and
$$
\partial_t(w_{\tau_j})\rightarrow \partial_t(|v|^{p-2}v)\;\; \text{in}\;\; L^q_{\text{loc}}([0,\infty); W^{-1,q}(\Omega)).
$$
Moreover, $v$ is a weak solution of \eqref{TruIVP} with $v(\cdot,0)=g$ that satisfies 
\begin{equation}\label{W1pMonForOneWS}
\int_\Omega|Dv(x,t)|^pdx\le \int_\Omega|Dv(x,s)|^pdx
\end{equation}
for almost every $t,s\in (0,\infty)$ with $t\ge s$.
\end{prop}
A useful estimate for our large time behavior considerations is as follows. 
\begin{cor}\label{ConvexAndW1pMonLemma}
Assume $v$ is a weak solution of \eqref{TruIVP} that satisfies \eqref{W1pMonForOneWS}. Then $t\mapsto \int_\Omega|v(x,t)|^pdx$ is 
convex and for $h>0$,
\begin{equation}\label{W1pBoundForOneWS}
\int_\Omega|Dv(x,t)|^pdx\le\left(\frac{p-1}{p}\right)\frac{e^{-p\left(\frac{\lambda_p}{p-1}\right)(t-h)}}{h}\int_\Omega|g|^pdx,\quad t\ge h.
\end{equation}
\end{cor}
\begin{proof}
The convexity assertion follows directly from \eqref{LpMon} and \eqref{W1pMonForOneWS}, so we will focus on establishing \eqref{W1pBoundForOneWS}.  For almost every $t\in (h,\infty)$, we have 
\begin{align*}
\int_\Omega|v(x,t-h)|^pdx &\ge \int_\Omega|v(x,t)|^pdx +\frac{d}{dt}\left(\int_\Omega|v(x,t)|^pdx\right)((t-h)-t) \\
&=\int_\Omega|v(x,t)|^pdx +\left(-\frac{p}{p-1}\int_\Omega|Dv(x,t)|^pdx\right)\cdot (-h)\\
&\ge h \frac{p}{p-1}\int_\Omega|Dv(x,t)|^pdx.
\end{align*}
By Corollary \ref{MonEthenLp}, 
\begin{align}\label{AEexpBound}
\int_\Omega|Dv(x,t)|^pdx&\le \left(\frac{p-1}{p}\right)\frac{1}{h}\int_\Omega|v(x,t-h)|^pdx \nonumber \\
& \le \left(\frac{p-1}{p}\right)\frac{e^{-p\left(\frac{\lambda_p}{p-1}\right)(t-h)}}{h}\int_\Omega|g|^pdx .
\end{align}
\par Let $t_0\ge h$ and select $t_m>h$ that converges to $t_0$ as $m\rightarrow\infty$, such that \eqref{AEexpBound} holds for each $t=t_m$. 
As $v(\cdot,t_m)\rightarrow v(\cdot, t_0)$ in $L^p(\Omega)$, it must be that $v(\cdot,t_m)\rightharpoonup v(\cdot, t_0)$ in $W^{1,p}_0(\Omega)$. 
Then weak convergence gives
$$
\int_\Omega|Dv(x,t_0)|^pdx\le \liminf_{m\rightarrow\infty}\int_\Omega|Dv(x,t_m)|^pdx \le \left(\frac{p-1}{p}\right)\frac{e^{-p\left(\frac{\lambda_p}{p-1}\right)(t_0-h)}}{h}\int_\Omega|g|^pdx,
$$
as claimed. 
\end{proof}

\section{Large time behavior}\label{LargeTimeSect}
We now set out to prove Theorem \ref{TruThm1}.  To this end, we will establish the following technical lemma. This result is important as it will help us identify the sign of various extremal functions that we will encounter when studying the large time limits of solutions of \eqref{TruIVP}.

\begin{lem}
Let $\ell,C>0$. There is $\delta=\delta(\ell,C)>0$ with the following property. Assume $v$ is a weak solution of \eqref{TruIVP} that satisfies

\begin{enumerate}[(i)]

\item $\ell\le \displaystyle\int_\Omega|v(x,0)|^pdx\le C$,

\item $\displaystyle\int_\Omega|v^+(x,0)|^pdx\ge \frac{1}{2}\ell$, and

\item 
$$
\frac{\||v(\cdot,0)|^{p-2}v(\cdot,0)\|^q_{L^{q}(\Omega)}}{\displaystyle\||v(\cdot,0)|^{p-2}v(\cdot,0)\|^q_{W^{-1,q}(\Omega)}}<\mu_p +\delta.
$$
\end{enumerate}
Then
$$
e^{p\left(\frac{\lambda_p}{p-1}\right)t}\int_\Omega|v^+(x,t)|^pdx\ge \frac{1}{2}\ell
$$
for $0\le t\le 1.$
\end{lem} 
\begin{rem}
By replacing $v$ with $-v$, the lemma also holds with $v^-$ replacing $v^+$.
\end{rem}
\begin{proof}
Suppose the assertion is false. Then there are $\ell,C>0$ and weak solutions $v_j$ of \eqref{TruIVP} that satisfy 

\begin{enumerate}[$(i)$]

\item $\ell\le \displaystyle\int_\Omega|v_j(x,0)|^pdx\le C$,

\item $\displaystyle\int_\Omega|v_j^+(x,0)|^pdx\ge \frac{1}{2}\ell$, and

\item 
$$
\frac{\||v_j(\cdot,0)|^{p-2}v_j(\cdot,0)\|^q_{L^{q}(\Omega)}}{\displaystyle\||v_j(\cdot,0)|^{p-2}v_j(\cdot,0)\|^q_{W^{-1,q}(\Omega)}}<\mu_p +\frac{1}{j},
$$
\end{enumerate}
while 
\begin{equation}\label{teeJayContra}
e^{p\left(\frac{\lambda_p}{p-1}\right)t_j}\int_\Omega|v_j^+(x,t_j)|^pdx< \frac{1}{2}\ell.
\end{equation}
for some $t_j\in [0,1]$.   Without any loss of generality, we may also assume that the sequence $(t_j)_{j\in \N}$ is convergent. In view of $(i)$, we can appeal to Proposition \ref{CompactTrudinger} to find a subsequence $(v_{j_k})_{k\in\N}$ and 
weak solution $v$ of \eqref{TruIVP} for which $v_{j_k}\rightarrow v$ in $C([0,1]; L^p(\Omega))$. 

\par From $(iii)$, $v(\cdot,0)$ is necessarily an extremal $u$ for the Poincar\'e inequality \eqref{Poincare}. By Corollary \ref{OnlyUniqueness}, 
$v(x,t)=e^{-(\lambda_p/(p-1))t}u$. Passing to the limit as $j=j_k\rightarrow\infty$ in 
$(i)$, $(ii)$ and  \eqref{teeJayContra} give
$$
\ell\le \int_\Omega|u|^pdx\quad \text{and}\quad \int_\Omega|u^+|^pdx= \frac{1}{2}\ell.
$$
Since every extremal of Poincar\'e's inequality \eqref{Poincare} does not change its sign in $\Omega$, it must be that $u>0$. It follows that $u=u^+$ and
$$
\frac{1}{2}\ell=\int_\Omega|u^+|^pdx =\int_\Omega|u|^pdx\ge \ell.
$$
This contradiction establishes the claim.
\end{proof}

\begin{cor}\label{positivityCorollary}
Let $\ell, C>0$ and select $\delta=\delta(\ell,C)$ from the previous lemma.  Suppose that $v$ is a weak solution of \eqref{TruIVP} which satisfies

\begin{enumerate}[(i)]

\item $\displaystyle e^{p\left(\frac{\lambda_p}{p-1}\right)t}\int_\Omega|v(x,t)|^pdx\ge \ell\;$   for   $\;t\ge 0$,

\item $\displaystyle\int_\Omega|v(x,0)|^pdx\le C$,

\item $\displaystyle\int_\Omega|v^+(x,0)|^pdx\ge \frac{1}{2}\ell$, and

\item 
$$
\frac{\||v(\cdot,0)|^{p-2}v(\cdot,0)\|^q_{L^{q}(\Omega)}}{\displaystyle\||v(\cdot,0)|^{p-2}v(\cdot,0)\|^q_{W^{-1,q}(\Omega)}}<\mu_p +\delta.
$$
\end{enumerate}
Then
\begin{equation}\label{vPlusLuhLowerBound}
e^{p\left(\frac{\lambda_p}{p-1}\right)t}\int_\Omega|v^+(x,t)|^pdx\ge \frac{1}{2}\ell
\end{equation}
for $t\ge 0$.
\end{cor}
\begin{rem}
As remarked above, we can replace $v$ with $-v$ to obtain an analogous statement for $v^-$.
\end{rem}
\begin{proof}
As $v$ satisfies hypotheses $(i), (ii),$ and $(iii)$ of the previous lemma,  \eqref{vPlusLuhLowerBound} holds for $t\in [0,1]$. Now define 
$$
v^1(x,t):=e^{\left(\frac{\lambda_p}{p-1}\right)}v(x,t+1), \quad (x,t)\in \Omega\times[0,\infty).
$$
By Corollary \ref{MonEthenLp}, 
$$
\int_\Omega|v^1(x,0)|^pdx=e^{p\left(\frac{\lambda_p}{p-1}\right)\cdot 1}\int_\Omega|v(x,1)|^pdx\le \int_\Omega|v(x,0)|^pdx\le C;
$$
and by Proposition \ref{WeakQuotientMon}, 
$$
\frac{\||v^1(\cdot,0)|^{p-2}v^1(\cdot,0)\|^q_{L^{q}(\Omega)}}{\displaystyle\||v^1(\cdot,0)|^{p-2}v^1(\cdot,0)\|^q_{W^{-1,q}(\Omega)}}\le
\frac{\||v(\cdot,0)|^{p-2}v(\cdot,0)\|^q_{L^{q}(\Omega)}}{\displaystyle\||v(\cdot,0)|^{p-2}v(\cdot,0)\|^q_{W^{-1,q}(\Omega)}}<\mu_p +\delta.
$$
Therefore, $v^1$ satisfies \eqref{vPlusLuhLowerBound} for $t\in [0,1]$ and consequently, $v$ satisfies \eqref{vPlusLuhLowerBound} holds for each $t\in [1,2]$.

\par Next we set
$$
v^k(x,t):=e^{\left(\frac{\lambda_p}{p-1}\right)k}v(x,t+k), \quad (x,t)\in \Omega\times[0,\infty)
$$
for $k\in \N$, $k\ge 2$. Observe that each $v^k$ is a weak solution of \eqref{TruIVP}. Using the argument above, it is straightforward to use mathematical induction to show that $v$ satisfies \eqref{vPlusLuhLowerBound} for $t$ belonging to the intervals $[k,k+1]$ for all $k\in \N$. We leave the details to the reader. 
\end{proof}
We now proceed to proving Theorem \ref{TruThm1}.

\begin{proof}[Proof of Theorem \ref{TruThm1}(i)]
Assume $v$ is a weak solution of \eqref{TruIVP} and set 
\begin{equation}\label{SLimit}
S:=\lim_{\tau\rightarrow\infty}e^{p\left(\frac{\lambda_p}{p-1}\right)\tau}\int_\Omega|v(x,\tau)|^pdx.
\end{equation}
Recall that this limit exists by Corollary \ref{MonEthenLp}. If $S=0$, then 
$\lim_{t\rightarrow \infty}e^{(\lambda_p/(p-1))t}v(\cdot,t)=0\in L^p(\Omega)$ and we conclude. So let us now suppose that $S>0$.  

\par Suppose $(s^k)_{k\in \N}$ is a sequence of positive numbers tending to $+\infty$ and define 
\begin{equation}\label{VeeKayDefn}
v^k(x,t):=e^{\left(\frac{\lambda_p}{p-1}\right)s_k}v(x,t+s_k), \quad (x,t)\in \Omega\times[0,\infty).
\end{equation}
Clearly, $v^k$ is a weak solution of \eqref{TruIVP} for each $k\in \N$. Moreover, 
$$
\int_\Omega|v^k(x,0)|^pdx=e^{p\left(\frac{\lambda_p}{p-1}\right)s_k}\int_\Omega|v(x,s_k)|^pdx\le \int_\Omega|g|^pdx.
$$
By Proposition \ref{CompactTrudinger}, there is a subsequence $(v^{k_j})_{j\in \N}$ and weak solution $v^\infty$ such that 
\begin{equation}\label{CompactApplicationLargeTimeProof}
v^{k_j}\rightarrow v^\infty\;\text{in}\; 
\begin{cases}
C_{\text{loc}}([0,\infty); L^p(\Omega))\\
L^p_{\text{loc}}([0,\infty); W^{1,p}_0(\Omega))
\end{cases}
\end{equation}
as $j\rightarrow\infty$. 

\par By \eqref{SLimit}, we have
\begin{align*}
S&=\lim_{j\rightarrow\infty}e^{p\left(\frac{\lambda_p}{p-1}\right)(t+s_{k_j})}\int_\Omega|v(x,t+s_{k_j})|^pdx\\
&=e^{p\left(\frac{\lambda_p}{p-1}\right)t}\lim_{j\rightarrow\infty}\int_\Omega|v^{k_j}(x,t)|^pdx\\
&=e^{p\left(\frac{\lambda_p}{p-1}\right)t}\int_\Omega|v^{\infty}(x,t)|^pdx
\end{align*}
for all $t> 0$.  Differentiating this equation in time (as in Lemma \ref{BasicMonLemma}) leads to 
\begin{equation}\label{vInfGroundState}
0=\left(\frac{p}{p-1}\right)e^{p\left(\frac{\lambda_p}{p-1}\right)t}\left(\lambda_p\int_\Omega|v^{\infty}(x,t)|^pdx- \int_\Omega|Dv^{\infty}(x,t)|^pdx \right)
\end{equation}
for almost every $t\ge 0$. As in our proof of Corollary \ref{ConvexAndW1pMonLemma}, \eqref{vInfGroundState} actually holds for every $t\ge 0$. 

\par In particular, 
$$
u:=\lim_{j\rightarrow\infty}e^{\left(\frac{\lambda_p}{p-1}\right)s_{k_j}}v(\cdot,s_{k_j})
$$
in $L^p(\Omega)$ for an extremal $u$ which satisfies $S=\int_\Omega|u|^pdx$. As the collection of extremals of the Poincar\'e inequality \eqref{Poincare} is one dimensional \cite{Saka}, $u$ is completely determined up to its sign.  We also have by Proposition \ref{WeakQuotientMon} that 
\begin{align*}
\lim_{t\rightarrow\infty}\frac{\||v(\cdot,t)|^{p-2}v(\cdot,t)\|^q_{L^{q}(\Omega)}}{\displaystyle\||v(\cdot,t)|^{p-2}v(\cdot,t)\|^q_{W^{-1,q}(\Omega)}}&=
\lim_{j\rightarrow\infty}\frac{\||v(\cdot,s_{k_j})|^{p-2}v(\cdot,s_{k_j})\|^q_{L^{q}(\Omega)}}{\displaystyle\||v(\cdot,s_{k_j})|^{p-2}v(\cdot,s_{k_j})\|^q_{W^{-1,q}(\Omega)}}\\
&=\frac{\||u|^{p-2}u\|^q_{L^{q}(\Omega)}}{\displaystyle\||u|^{p-2}u\|^q_{W^{-1,q}(\Omega)}}\\
&=\mu_p.
\end{align*}

\par As previously mentioned, it must be that either $u>0$ or $u<0$ in $\Omega$. Without loss of generality, we will suppose that $u>0$. Set $\ell:=S>0$, $C:=\int_\Omega|g|^pdx$ and choose $\delta=\delta(\ell,C)$ as in Corollary \ref{positivityCorollary}.  From our analysis above, there 
exist a $j^*\in \N$ such that $v^{k_j}$ satisfies hypotheses $(ii)-(iv)$ in Corollary \ref{positivityCorollary} for each $j\in \N$ with $j\ge j^*$. To verify hypothesis $(i)$, we only need to recall that $S$ is the infimum of $e^{p\left(\lambda_p/(p-1)\right)\tau}\int_\Omega|v(x,\tau)|^pdx$ over $\tau>0$. Therefore, 
$$
e^{p\left(\frac{\lambda_p}{p-1}\right)t}\int_\Omega|v^k(x,t)|^pdx\ge S
$$
for every $k\in \N$ and $t\ge 0$. It then follows that 
\begin{equation}\label{vKayJayPositivitiy}
e^{p\left(\frac{\lambda_p}{p-1}\right)(t+s_{k_j})}\int_\Omega|v^+(x,t+s_{k_j})|^pdx\ge \frac{1}{2}S
\end{equation}
for every $t\ge 0$ and each $j\ge j^*$.

\par Now suppose there is another sequence of positive number $(t_m)_{m\in \N}$ that increase to $+\infty$ for which 
$$
\lim_{m\rightarrow \infty}e^{\left(\frac{\lambda_p}{p-1}\right)t_m}v(\cdot,t_m)=-u
$$
in $L^p(\Omega)$.  Select a subsequence $(t_{m_j})_{j\in \N}$ such that $t_{m_j}>s_{k_j}$ for all $j\in \N$.  Substituting $t=t_{m_j}-s_{k_j}$ in \eqref{vKayJayPositivitiy} gives
$$
e^{p\left(\frac{\lambda_p}{p-1}\right)t_{m_j}}\int_\Omega|v^+(x,t_{m_j})|^pdx\ge \frac{1}{2}S.
$$
Sending $j\rightarrow\infty$ leads to 
$$
\int_\Omega|(-u)^+|^pdx\ge \frac{1}{2}S,
$$
which is a contradiction to $u$ being a positive function.  Finally, as $S$ is independent of the sequence $(s_k)_{k\in \N}$, the full limit $\lim_{t\rightarrow \infty}e^{\left(\lambda_p/(p-1)\right)t}v(\cdot,t)=u$ exists in $L^p(\Omega)$.
\end{proof}
\begin{proof}[Proof of Theorem \ref{TruThm1}(ii)]
Assume $v$ is a weak solution of \eqref{TruIVP} that satisfies \eqref{W1pMonForOneWS}.  Let $(s^k)_{k\in \N}$ be a sequence of positive numbers tending to $\infty$ and define $v^k$ by \eqref{VeeKayDefn} for each $k\in \N$. By part $(i)$ of this proof, we have that 
\begin{equation}\label{vkStrongLim}
\lim_{k\rightarrow \infty}v^k(\cdot,t)=e^{-\left(\frac{\lambda_p}{p-1}\right)t}u
\end{equation}
exists in $L^p(\Omega)$ for each time $t\ge 0$; here $u$ is an extremal of the Poincar\'e inequality \eqref{Poincare}.  Applying \eqref{W1pBoundForOneWS} to $v^k$, we see that 
$(v^k(\cdot,t))_{k\in \N}$ is bounded in $W^{1,p}_0(\Omega)$ for each $t\ge 0$. Therefore, \eqref{vkStrongLim} holds weakly in $W^{1,p}_0(\Omega)$ for all $t\ge 0$. 

\par By \eqref{CompactApplicationLargeTimeProof}, we also have that \eqref{vkStrongLim} holds strongly in $W^{1,p}_0(\Omega)$ for almost every $t\ge 0$ for a subsequence $(v^{k_j})_{j\in \N}$.  Since $v$ satisfies \eqref{W1pMonForOneWS}, $[0,\infty)\ni t\mapsto \int_\Omega|Dv^k(x,t)|^pdx$ is nonincreasing for each $k\in \N$.  By Helly's Theorem (Lemma 3.3.3 in \cite{AGS}), there is a subsequence (again labeled) $(v^{k_j})_{j\in \N}$ such that the limit 
$$
h(t):=\lim_{j\rightarrow \infty}\int_\Omega|Dv^{k_j}(x,t)|^pdx
$$
holds for every $t\ge 0$.  As noted above, 
$$
h(t)=\int_\Omega|e^{-\left(\frac{\lambda_p}{p-1}\right)t}Du|^pdx\quad\text{a.e.}\; t\ge0
$$
and
$$
h(t)\ge \int_\Omega|e^{-\left(\frac{\lambda_p}{p-1}\right)t}Du|^pdx\quad \text{all}\; t\ge 0.
$$

\par Repeating the steps of part 4 of our proof of Proposition \ref{CompactTrudinger}, we are able to conclude  
$$
\lim_{j\rightarrow \infty}\int_\Omega|Dv^{k_j}(x,t)|^pdx= \int_\Omega|e^{-\left(\frac{\lambda_p}{p-1}\right)t}Du|^pdx
$$
for every $t\ge 0$.  As \eqref{vkStrongLim} holds weakly in $W^{1,p}_0(\Omega)$ at $t=0$, we have 
$$
\lim_{j\rightarrow \infty}e^{\left(\frac{\lambda_p}{p-1}\right)s_{k_j}}v(\cdot,s_{k_j})=u
$$
in $W^{1,p}_0(\Omega)$. Therefore, for every sequence of positive numbers $(s^k)_{k\in \N}$ tending to $\infty$,
$(e^{\left((\lambda_p/p-1)\right)s_{k}}v(\cdot,s_{k}))_{k\in \N}$ has a 
subsequence that converges to $u$. It follows that 
$$
\lim_{t\rightarrow \infty}e^{\left(\frac{\lambda_p}{p-1}\right)t}v(\cdot,t)=u
$$
in $W^{1,p}_0(\Omega)$. 
Finally, if $u$ does not vanish identically, then $v(\cdot,t)$ does not vanish identically for all $t\ge 0$ and
$$
\lim_{t\rightarrow \infty}\frac{\int_\Omega|Dv(x,t)|^pdx}{\int_\Omega|v(x,t)|^pdx}=\frac{\int_\Omega|Du|^pdx}{\int_\Omega|u|^pdx}=\lambda_p.
$$
\end{proof}
Now we will comment briefly on how to rule out degeneracy in the limit described in Theorem \ref{TruThm1}. 
Our remarks will be mostly based on the following observation.  Suppose $f\in L^p(\Omega)$ and 
$u\in W^{1,p}_0(\Omega)$ is a weak solution of the boundary value problem 
\begin{equation}\label{OneStep}
\begin{cases}
\displaystyle\frac{{\cal J}_p(u)- {\cal J}_p(f)}{\tau}=\Delta_pu\quad &\text{in}\;\Omega \\
\hspace{.95in}u=0\quad &\text{on}\;\partial\Omega.
\end{cases}
\end{equation}
Also assume that $\varphi$ is an extremal for the Poincar\'e inequality \eqref{Poincare}. If $f\ge \varphi$, then 
$$
u\ge (1+\lambda_p\tau)^{-\frac{1}{p-1}}\varphi
$$
in $\Omega$. This inequality follows by weak comparison. Indeed, the function $(1+\lambda_p\tau)^{-\frac{1}{p-1}}\varphi$ is a subsolution of the elliptic equation in \eqref{OneStep} and agrees with $u$ on $\partial\Omega$. 

\begin{prop}\label{NondegeneracyProp}
Assume that $v$ is a weak solution of \eqref{TruIVP} as described in Proposition \ref{ExistenceProp}. If 
 $\varphi$ is an extremal for the Poincar\'e inequality \eqref{Poincare} and $g\ge \varphi$, then 
\begin{equation}\label{LowerBoundRescaled}
 v(\cdot,t)\ge e^{-\left(\frac{\lambda_p}{p-1}\right)t}\varphi
\end{equation}
for each $t\ge 0$.
\end{prop}
\begin{proof}
Let $(u^k)_{k\in \N}$ be a solution of the implicit time scheme \eqref{ImplicitTimeScheme} with $u^0=g$. Iterating 
\eqref{OneStep}, we find
$$
u^k\ge (1+\lambda_p\tau)^{-\frac{k}{p-1}}\varphi
$$
for each $k\in \N$. Therefore, for $t\in ((k-1)\tau,k\tau]$
\begin{equation}\label{tauLowerBoundest}
v_\tau(\cdot, t)\ge (1+\lambda_p\tau)^{-\frac{k}{p-1}}\varphi.
\end{equation}
Now choose $\tau=\tau_j$ as in Proposition \ref{ExistenceProp} and select $k=k_j\in \N$ so that $t\in ((k_j-1)\tau_j,k_j\tau_j]$. 
With these choices, we can send $j\rightarrow\infty$ in \eqref{tauLowerBoundest} and deduce \eqref{LowerBoundRescaled}.  
\end{proof}
It is now immediate that if the initial condition $g$ in \eqref{TruIVP} is larger than a positive extremal for the Poincar\'e inequality \eqref{Poincare}, 
then the limit described in Theorem \ref{TruThm1} does not vanish identically for a weak solution as described in Proposition \ref{ExistenceProp}.  Likewise, if $g$ is smaller than a negative extremal, there is no
degeneracy.   A simple choice of an initial condition that ensures nondegeneracy is 
$$
g\equiv 1.
$$
As long as $\partial\Omega$ is smooth, any positive extremal $\varphi_0$ for the Poincar\'e inequality \eqref{Poincare} is in fact continuous on $\overline{\Omega}$ \cite{Lin, Saka}.
In particular, there is $\epsilon>0$ such that $\epsilon \varphi_0\le 1$ in $\Omega$. So we can pick $\varphi=\epsilon\varphi_0$, $g\equiv 1$ and produce a weak solution $v$ 
that satisfies \eqref{LowerBoundRescaled}.

\section{Robin boundary condition}\label{RBC}
We will now consider the large time behavior of weak solutions of the initial value problem 
for Trudinger's equation with a Robin boundary condition \eqref{TruRBC}.  Our goal is primarily to explain how our analysis of the initial value problem \eqref{TruIVP} studied in the previous sections carries over in this setting.  Therefore, we will present a streamlined treatment of the initial value problem \eqref{TruRBC}. Throughout this section, we will assume that $\partial\Omega$ is $C^1$ with outward unit normal $\nu$.  

\par In view of the Poincar\'e inequality \eqref{PoincareRBC}, we will equip $W^{1,p}(\Omega)$ with the norm
$$
\|u\|_{W^{1,p}(\Omega)}:=\left(\int_\Omega|Du|^pdx+\beta \int_{\partial\Omega}|Tu|^pd\sigma\right)^{1/p}.
$$
We will also make use of the fact that the following boundary value problem 
\begin{equation}\label{RBCPoissonEqn}
\begin{cases}
\hspace{1.35in}-\Delta_pu=f\quad&\text{in}\;\Omega\\
|Du|^{p-2}Du\cdot \nu +\beta|u|^{p-2}u=0\quad&\text{on}\;\partial\Omega
\end{cases}
\end{equation}
has a unique weak solution for each $f\in (W^{1,p}(\Omega))^*$. Recall that a weak solution of \eqref{RBCPoissonEqn} is defined to be a function $u\in W^{1,p}(\Omega)$ that satisfies
$$
\int_\Omega|Du|^{p-2}Du\cdot D\psi dx+\beta \int_{\partial\Omega}(|Tu|^{p-2}Tu)\;T\psi d\sigma=\langle f,\psi\rangle
$$
for each $\psi\in W^{1,p}(\Omega)$.  

\par It is then natural to consider the operator ${\cal A}_p: W^{1,p}(\Omega)\rightarrow (W^{1,p}(\Omega))^*$ given by
$$
\langle {\cal A}_pu,\psi\rangle:=\int_\Omega|Du|^{p-2}Du\cdot D\psi dx+\beta \int_{\partial\Omega}(|Tu|^{p-2}Tu)\;T\psi d\sigma
$$
($u, \psi\in W^{1,p}(\Omega))$.  Note that $\{{\cal A}_pu\}\subset (W^{1,p}(\Omega))^*$ is the subdifferential of $\frac{1}{p}\|\cdot\|_{W^{1,p}(\Omega)}^p$ at $u$, and so ${\cal A}_p$ is strictly monotone. It is also routine to check that if $u\in C^2(\overline{\Omega})$ satisfies 
$$
|Du|^{p-2}Du\cdot \nu +\beta|u|^{p-2}u=0\quad\text{on}\;\partial\Omega,
$$
then ${\cal A}_pu=-\Delta_pu$.  It turns out that ${\cal A}_p$ can be used to analyze the initial value problem for Trudinger's equation with a Robin boundary condition the same way $-\Delta_p$ defined in \eqref{W1pzeroPlaplace} was employed in our analysis of Trudinger's equation with a Dirichlet boundary condition.

\par Direct computation also leads to the identity
$$
\|{\cal A}_pu\|_{(W^{1,p}(\Omega))^*}=\|u\|_{W^{1,p}(\Omega)}^{p-1} \quad (u\in W^{1,p}(\Omega)).
$$
This identity can be used to derive the following dual Poincar\'e inequality 
\begin{equation}\label{RBCdualPoincare}
\mu_p\|f\|^q_{(W^{1,p}(\Omega))^*}\le \int_\Omega|f|^qdx\quad (f\in L^q(\Omega)),
\end{equation}
where $\mu_p:=\lambda_p^{\frac{1}{p-1}}$ (see Appendix \ref{AppDPI}). Here, and for the rest of this section, $\lambda_p$ is the constant in the Poincar\'e inequality \eqref{PoincareRBC}.

\par Notice that for any smooth solution of \eqref{TruRBC} and $t>0$, we have
\begin{equation}\label{RBCLpGoDown}
\frac{d}{dt}\frac{1}{p}\int_\Omega|v(x,t)|^pdx=-\frac{1}{p-1}\left(
\int_\Omega|Dv(x,t)|^pdx+\beta \int_{\partial\Omega}|(Tv)(x,t)|^pd\sigma\right).
\end{equation}
Therefore, 
$$
\frac{1}{p}\int_\Omega|v(x,t)|^pdx+\frac{1}{p-1}\int^t_0\left(
\int_\Omega|Dv|^pdx+\beta \int_{\partial\Omega}|Tv|^pd\sigma\right)ds
=\frac{1}{p}\int_\Omega|g(x)|^pdx
$$
for each $t\ge0$. These computations motivate the following definition. 
\begin{defn}
Assume $g\in L^p(\Omega)$. A {\it weak solution} of \eqref{TruRBC} is a function $v:\Omega\times[0,\infty)\rightarrow \R$ that satisfies:
$(i)$
$$
v\in L^\infty([0,\infty);L^p(\Omega))\cap L^p([0,\infty); W^{1,p}(\Omega));
$$
$(ii)$
$$
\int^\infty_0\int_\Omega |v|^{p-2}v\psi_tdxdt =\int^\infty_0\int_\Omega |Dv|^{p-2}Dv\cdot D\psi dxdt + \beta \int^\infty_0\int_{\partial\Omega} |Tv|^{p-2}Tv\;\psi d\sigma dt
$$
for each $\psi\in C^\infty_c\left(\overline{\Omega}\times(0,\infty)\right)$; and $(iii)$
$$
v(\cdot,0)=g.
$$
\end{defn}
The fundamental continuity and monotonicity properties of weak solutions are  summarized below. 
\begin{prop}
Assume that $v$ is a weak solution of \eqref{TruRBC}.  Then $v$ has the following properties.
\begin{enumerate}[(i)]

\item $|v|^{p-2}v\in AC^q_{\text{loc}}([0,\infty); (W^{1,p}(\Omega))^*)$.

\item $[0,\infty) \ni t\mapsto \displaystyle\int_\Omega|v(x,t)|^pdx$ is locally absolutely continous.

\item \eqref{RBCLpGoDown} holds for almost every $t>0$.

\item $[0,\infty) \ni t\mapsto e^{p\left(\frac{\lambda_p}{p-1}\right)t}\displaystyle\int_\Omega|v(x,t)|^pdx$ is nonincreasing.

\item $v:[0,\infty)\rightarrow L^p(\Omega)$ is bounded and uniformly continuous.

\item If $v(\cdot,t)\not\equiv 0$ for $t\ge 0$, then 
$$
[0,\infty) \ni t\mapsto\frac{\||v(\cdot,t)|^{p-2}v(\cdot,t)\|^q_{L^{q}(\Omega)}}{\displaystyle\||v(\cdot,t)|^{p-2}v(\cdot,t)\|^q_{ (W^{1,p}(\Omega))^*}}
$$
is nonincreasing. 
\end{enumerate}
\end{prop}
\begin{rem}
Similar to how we argued in Remark \ref{StrongMonRemark}, it can be shown that
$$
\frac{d}{dt}\left\{\frac{\displaystyle\int_\Omega|Dv(x,t)|^pdx+\beta\int_{\partial\Omega}|v(x,t)|^pd\sigma}{\displaystyle\int_\Omega|v(x,t)|^pdx}\right\}\le 0.
$$
for any smooth, nonvanishing solution $v$ of \eqref{TruRBC}.
\end{rem}
Weak solutions of \eqref{TruRBC} have compactness properties analogous to the compactness detailed in Proposition \ref{CompactTrudinger}. In order to write a corresponding statement for weak solutions \eqref{TruRBC}, we would only need
to change $W^{1,p}_0(\Omega)$ to $W^{1,p}(\Omega)$ and change $W^{-1,q}(\Omega)$ to $(W^{1,p}(\Omega))^*.$ Moreover, a weak solution can be constructed using the following implicit time scheme: $u^0=g$,
$$
\begin{cases}
\hspace{.75in}\displaystyle\frac{{\cal J}_p(u^k)- {\cal J}_p(u^{k-1})}{\tau}=\Delta_pu^k\quad &\text{in}\;\Omega \\
|Du^k|^{p-2}Du^k\cdot \nu +\beta|u^k|^{p-2}u^k=0\quad&\text{on}\;\partial\Omega
\end{cases}
$$
for $k\in \N$. Employing the ideas used to prove Proposition \ref{ExistenceProp}, we can show there is a weak solution $v$ of \eqref{TruRBC}  with 
$$
\int_\Omega|Dv(x,t)|^pdx+\beta\int_{\partial\Omega}|(Tv)(x,t)|^pd\sigma\le \int_\Omega|Dv(x,s)|^pdx+\beta\int_{\partial\Omega}|(Tv)(x,s)|^pd\sigma
$$
for almost every $t,s\in (0,\infty)$ with $t\ge s$.

\par We are now in position to make use of the methods of the previous sections and characterize the large time behavior of weak solutions of \eqref{TruRBC}.  However, we will not give a detailed proof as our argument follows closely to our proof of Theorem \ref{TruThm1}.  We only mention that in order to adapt our proof of Theorem \ref{TruThm1}, we use that extremals of the Poincar\'e inequality \eqref{PoincareRBC} exist, are weak solutions of the PDE
$$
\begin{cases}
\hspace{1.35in}-\Delta_pu=\lambda_p|u|^{p-2}u\quad&\text{in}\;\Omega\\
|Du|^{p-2}Du\cdot \nu +\beta|u|^{p-2}u=0\quad&\text{on}\;\partial\Omega,
\end{cases}
$$
do not change sign in $\Omega$, and the ratio of any two nonvanishing extremals is constant \cite{BelKaw, Bucur, DaiFu, KawLin}.  Our main result regarding \eqref{TruRBC} is as follows.

\begin{customthm}{2}\label{TruThmRBC} 
(i) Assume $v$ is a weak solution of \eqref{TruRBC}.  Then the limit 
\begin{equation}\label{RBCscaledLim}
u:=\lim_{t\rightarrow\infty}e^{\left(\frac{\lambda_p}{p-1}\right)t}v(\cdot,t)
\end{equation}
exists in $L^p(\Omega)$ and $u$ is extremal for \eqref{PoincareRBC}. If $u\not\equiv 0$, then
$v(\cdot,t)\not\equiv 0$ for all $t\ge 0$ and 
$$
\mu_p=\lim_{t\rightarrow\infty}\frac{\||v(\cdot,t)|^{p-2}v(\cdot,t)\|^q_{L^q(\Omega)}}{\||v(\cdot,t)|^{p-2}v(\cdot,t)\|^q_{ (W^{1,p}(\Omega))^*}}.
$$ 
(ii) There is a weak solution $v$ of \eqref{TruRBC} such that the limit \eqref{RBCscaledLim}
exists in $W^{1,p}(\Omega)$.  If $u\not\equiv 0$,
$$
\lambda_p=\lim_{t\rightarrow\infty}\frac{\displaystyle\int_\Omega|Dv(x,t)|^pdx+\beta\int_{\partial\Omega}|(Tv)(x,t)|^pd\sigma}{\displaystyle\int_\Omega|v(x,t)|^pdx}.
$$ 
\end{customthm}

\section{Neumann boundary condition}\label{NBC}
Now we will study the initial value problem \eqref{TruNBC}, which is the analog of \eqref{TruIVP} with a Neumann boundary condition. As mentioned in the introduction, we will assume that $\partial\Omega$ is $C^1$ with outward unit normal field $\nu$.  As with the previous initial value problems we have considered so far, our aim is to deduce the large time behavior of solutions of \eqref{TruNBC}. However, unlike our study of \eqref{TruRBC}, our treatment of \eqref{TruNBC} is not a direct generalization of our analysis of \eqref{TruIVP}.  We will need to make use of one of our prior results (Theorem 1.3 of \cite{HynLin}) on the large time behavior of general curves of maximal slope in Banach spaces.

\subsection{Preliminaries}
 A distinguishing feature of the initial value problem \eqref{TruNBC} is that the integral $\int_\Omega|v|^{p-2}vdx$ is conserved along the flow. Indeed, if $v$ is a smooth solution of \eqref{TruNBC}, then 
\begin{align*}
\frac{d}{dt}\int_\Omega|v|^{p-2}vdx&=\int_\Omega\partial_t(|v|^{p-2}v)dx\\
&=\int_\Omega\Delta_pv dx\\
&=\int_{\partial\Omega}|Dv|^{p-2}Dv\cdot \nu d\sigma\\
&=0.
\end{align*}
A simplifying assumption that we will make is that $\int_\Omega|g|^{p-2}gdx=0$, which in turn gives $\int_\Omega|v|^{p-2}vdx=0$ for all later times. Therefore, the theory we present has to accommodate this constraint.

\par To this end, it will be convenient to make use of the Poincar\'e inequality \eqref{PoincareNBC}.  Similar to the Poincar\'e inequalities referenced in this paper, extremal functions exist and satisfy a boundary value problem which takes the form
$$
\begin{cases}
\hspace{.55in}-\Delta_pu=\lambda_p|u|^{p-2}u\quad&\text{in}\;\Omega\\
|Du|^{p-2}Du\cdot \nu =0\quad&\text{on}\;\partial\Omega.
\end{cases}
$$
In this section, $\lambda_p$ is the optimal constant in \eqref{PoincareNBC}.  However, a major difference between \eqref{PoincareNBC} with the other Poincar\'e inequalities studied in this paper is that extremals do not possess a definite sign in $\Omega$ nor are in general unique up to a multiplicative constant. These differences are precisely what lead us to use different techniques when studying the large time behavior of \eqref{TruNBC}.

\par As with our previous arguments, we will need to employ a Poincar\'e inequality that is dual to \eqref{PoincareNBC}.  This inequality will involve ${\cal C}$, the collection of measurable functions on $\Omega$ that are constant almost everywhere.  In particular, a space that will be of interest for us is the annihilator of ${\cal C}$ 
\begin{equation}\label{Xspace}
{\cal C}^\perp:=\{f\in (W^{1,p}(\Omega))^*: \left.f\right|_{{\cal C}}=0\}.
\end{equation}
For each $f\in {\cal C}^\perp$, the Neumann problem
\begin{equation}\label{NeumannProblem}
\begin{cases}
\hspace{.55in}-\Delta_pu=f\quad&\text{in}\;\Omega\\
|Du|^{p-2}Du\cdot \nu =0\quad&\text{on}\;\partial\Omega
\end{cases}
\end{equation}
has at least one weak solution $u\in W^{1,p}(\Omega)$.  That is, there is at least one $u\in W^{1,p}(\Omega)$ that satisfies 
\begin{equation}\label{WeakSolnNeumann}
\int_\Omega |Du|^{p-2}Du\cdot D\phi dx =\langle f, \phi\rangle 
\end{equation}
for each $\phi\in W^{1,p}(\Omega)$.  

\par It is not difficult to see that a weak solution of \eqref{NeumannProblem} is determined uniquely up to an additive constant. Consequently, there is only one weak solution of \eqref{NeumannProblem} that satisfies \eqref{pAverageZero}.  So if we set
\begin{equation}\label{NeumannSpaceS}
{\cal S}:=\{u\in W^{1,p}(\Omega): u\;\text{satisfies \eqref{pAverageZero}}\},
\end{equation}
we see that ${\cal A}_p:{\cal S}\rightarrow {\cal C}^\perp$ defined by
\begin{equation}\label{NeumannPLaclace}
\langle {\cal A}_pu,\psi\rangle:=\int_\Omega|Du|^{p-2}Du\cdot D\psi dx
\end{equation}
$(u\in {\cal S},\psi\in W^{1,p}(\Omega))$ is a bijection.  We also note that if $u\in C^2(\overline{\Omega})$ satisfies the Neumann condition 
$$
|Du|^{p-2}Du\cdot \nu =0\quad\text{on}\;\partial\Omega,
$$
then ${\cal A}_pu=-\Delta_pu$. 

\par The above definition of ${\cal A}_p$ allows us to equip the space ${\cal C}^\perp$ with a convenient norm 
\begin{equation}\label{normXdefn}
\|f\|_{{\cal C}^\perp}:=\langle f, {\cal A}_p^{-1}f\rangle^{1/q}.
\end{equation}
In Appendix \ref{normX}, we show ${\cal C}^\perp$ is a reflexive Banach space under this norm. The arguments given in Appendix \ref{AppDPI} will additionally imply that the dual Poincar\'e inequality 
\begin{equation}\label{NBCDualPoincare}
\mu_p\|f\|^q_{{\cal C}^\perp}\le \int_\Omega |f|^qdx
\end{equation}
holds for each $f\in L^q(\Omega)$ with $\int_\Omega fdx=0$.  Here $\mu_p:=\lambda_p^{\frac{1}{p-1}}$, 
and equality holds if and only if $f=|u|^{p-2}u$ and $u$ is extremal for \eqref{PoincareNBC}.

\subsection{Weak solutions}
For a smooth solution of \eqref{TruNBC} and $t>0$, we calculate
\begin{equation}\label{NBCLpGoDown}
\frac{d}{dt}\frac{1}{p}\int_\Omega|v(x,t)|^pdx=-\frac{1}{p-1}
\int_\Omega|Dv(x,t)|^pdx.
\end{equation}
As a result, 
$$
\frac{1}{p}\int_\Omega|v(x,t)|^pdx+\frac{1}{p-1}\int^t_0
\int_\Omega|Dv(x,s)|^pdxds=\frac{1}{p}\int_\Omega|g(x)|^pdx
$$
for $t\ge0$. This observation leads to the following definition. 
\begin{defn}
Assume $g\in L^p(\Omega)$ satisfies \eqref{pAverageZero}. A {\it weak solution} of \eqref{TruNBC} is a function $v:\Omega\times[0,\infty)\rightarrow \R$ that fulfills:
$(i)$
$$
v\in L^\infty([0,\infty);L^p(\Omega))\cap L^p([0,\infty); W^{1,p}(\Omega));
$$
$(ii)$ $v(\cdot,t)$ satisfies \eqref{pAverageZero} for all $t\ge 0$;\\
$(iii)$
\begin{equation}\label{NBCWeakSoln}
\int^\infty_0\int_\Omega |v|^{p-2}v\psi_tdxdt=\int^\infty_0\int_\Omega |Dv|^{p-2}Dv\cdot D\psi dxdt
\end{equation}
for each $\psi\in C_c^\infty\left(\overline{\Omega}\times(0,\infty)\right)$; and $(iii)$
$$
v(\cdot,0)=g.
$$
\end{defn}

\par We now list the relevant properties of weak solutions of \eqref{TruNBC}. 
\begin{prop}
Assume that $v$ is a weak solution of \eqref{TruNBC}.  Then $v$ has the following properties.
\begin{enumerate}[(i)]

\item $|v|^{p-2}v\in AC^q_{\text{loc}}([0,\infty); {\cal C}^\perp)$.

\item $[0,\infty) \ni t\mapsto \displaystyle\int_\Omega|v(x,t)|^pdx$ is locally absolutely continous.

\item \eqref{NBCLpGoDown} holds for almost every $t>0$.

\item $[0,\infty) \ni t\mapsto e^{p\left(\frac{\lambda_p}{p-1}\right)t}\displaystyle\int_\Omega|v(x,t)|^pdx$ is nonincreasing.

\item $v:[0,\infty)\rightarrow L^p(\Omega)$ is bounded and uniformly continuous.

\item If $v(\cdot,t)\not\equiv 0$ for $t\ge 0$, then 
$$
[0,\infty) \ni t\mapsto\frac{\||v(\cdot,t)|^{p-2}v(\cdot,t)\|^q_{L^{q}(\Omega)}}{\displaystyle\||v(\cdot,t)|^{p-2}v(\cdot,t)\|^q_{ {\cal C}^\perp}}
$$
is nonincreasing. 
\end{enumerate}
\end{prop}
\begin{rem}
Similar to Remark \ref{StrongMonRemark}, we can verify 
$$
\frac{d}{dt}\left\{\frac{\displaystyle\int_\Omega|Dv(x,t)|^pdx}{\displaystyle\int_\Omega|v(x,t)|^pdx}\right\}\le 0.
$$
for any smooth, nonvanishing solution $v$ of \eqref{TruRBC}.
\end{rem}

Weak solutions of \eqref{TruNBC} have compactness properties similar to the compactness presented in Proposition \ref{CompactTrudinger}. In order to phrase an analogous theorem for weak solutions \eqref{TruNBC}, one simply has to exchange $W^{1,p}_0(\Omega)$ with $W^{1,p}(\Omega)$ and substitute  $W^{-1,q}(\Omega)$ with ${\cal C}^\perp.$ Further, a weak solution of  \eqref{TruNBC} can be designed using the following implicit time scheme: set $u^0=g$,  find $u^k\in W^{1,p}(\Omega)$ satisfying \eqref{pAverageZero} and
$$
\begin{cases}
\displaystyle\frac{{\cal J}_p(u^k)- {\cal J}_p(u^{k-1})}{\tau}=\Delta_pu^k\quad &\text{in}\;\Omega \\
\hspace{.2in}|Du^k|^{p-2}Du^k\cdot \nu=0\quad&\text{on}\;\partial\Omega
\end{cases}
$$
for each $k\in \N$. The same ideas presented in our proof of Proposition \ref{ExistenceProp}, can be used to show there is a weak solution $v$ of \eqref{TruNBC} that satisfies 
$$
\int_\Omega|Dv(x,t)|^pdx\le \int_\Omega|Dv(x,s)|^pdx
$$
for almost every $t,s\in (0,\infty)$ with $t\ge s$.

\subsection{Large time limit}
We are now ready to present our large time limit result for solutions of \eqref{TruNBC}.  Our strategy will be different than how we approached the previous initial value problems because the Neumann eigenvalue problem does in general not have signed solutions nor solutions that are unique up to multiplication by constants. In particular, our proof Corollary \ref{positivityCorollary} and Theorem \ref{TruThm1} part $(i)$ cannot be directly adapted to this setting.  Instead we will use a general result about the large time behavior of doubly nonlinear evolutions we derived in our previous work \cite{HynLin}. We did not pursue this approach throughout the entirety of this paper as it relies on technical results and because we wanted to prove the results in this paper in an accessible fashion. 

\begin{customthm}{3}\label{TruThmNBC} 
Suppose that the ratio of any two extremals of \eqref{PoincareNBC} that do not vanish identically is constant. 
\\ (i) Assume $v$ is a weak solution of \eqref{TruNBC}.  Then the limit 
\begin{equation}\label{scaledLimNBC}
u:=\lim_{t\rightarrow\infty}e^{\left(\frac{\lambda_p}{p-1}\right)t}v(\cdot,t)
\end{equation}
exists in $L^p(\Omega)$, and $u$ is extremal for \eqref{PoincareNBC}. If $u\not\equiv 0$, then 
\begin{equation}\label{RatioLimNBC}
\mu_p=\lim_{t\rightarrow\infty}\frac{\||v(\cdot,t)|^{p-2}v(\cdot,t)\|^q_{L^q(\Omega)}}{\||v(\cdot,t)|^{p-2}v(\cdot,t)\|^q_{{\cal C}^\perp}}.
\end{equation} 
(ii) There is a weak solution $v$ of \eqref{TruNBC} such that the limit \eqref{scaledLimNBC}
exists in $W^{1,p}(\Omega)$.  If $u\not\equiv 0$,  
$$
\lambda_p=\lim_{t\rightarrow\infty}\frac{\displaystyle\int_\Omega|Dv(x,t)|^pdx}{\displaystyle\int_\Omega|v(x,t)|^pdx}.
$$
\end{customthm}
\begin{proof}
We will only prove part $(i)$ since part $(ii)$ can be readily adapted from the proof of part $(ii)$ of Theorem
\ref{TruThm1}, once part $(i)$ has been established.  Note that the definition of ${\cal A}_p$ in \eqref{NeumannPLaclace} and the weak solution 
condition \eqref{NBCWeakSoln} imply
$$
\partial_t(|v(\cdot,t)|^{p-2}v(\cdot,t))+{\cal A}_p(v(\cdot,t))=0\quad \text{a.e.}\;\; t>0.
$$
Therefore, if we set $w:=|v|^{p-2}v$, then
\begin{equation}\label{wEqnNBC}
{\cal A}_p^{-1}(\partial_tw(\cdot,t))+|w(\cdot,t)|^{q-2}w(\cdot,t)=0\quad  \text{a.e.}\;\; t>0.
\end{equation}
We will now interpret the flow \eqref{wEqnNBC} as an abstract doubly nonlinear evolution.

\par To this end, we first note that  for each $f\in {\cal C}^\perp$, ${\cal A}_p^{-1}f$ belongs to the subdifferential of $\Psi=\frac{1}{q}\|\cdot\|_{{\cal C}^\perp}^q$ at $f$ (see Remark \ref{SubDiffRem}). For a given $f\in {\cal C}^\perp$, we also define
$$
\Phi(f):=
\begin{cases}
\frac{1}{q}\int_\Omega|f|^qdx, &\quad  f\in L^q(\Omega),\; \int_\Omega fdx=0\\
+\infty, &\quad \text{otherwise}.
\end{cases}
$$
Equation \eqref{wEqnNBC} can now be rewritten as the doubly nonlinear flow
$$
\partial\Psi(\partial_tw(\cdot,t))+\partial\Phi(w(\cdot,t))\ni 0,
$$
where $w\in AC^q_\text{loc}([0,\infty); {\cal C}^\perp)$.  Moreover, the dual Poincar\'e inequality \eqref{NBCDualPoincare} can be written 
\begin{equation}\label{NBCDualPoincare2}
\mu_p\Psi(f)\le \Phi(f), \quad f\in {\cal C}^\perp.
\end{equation}

\par With this reinterpretation of the flow \eqref{TruNBC}, we can now apply Theorem 1.3 of \cite{HynLin}.  This result implies that there is $f\in {\cal C}^\perp$ for which equality holds in \eqref{NBCDualPoincare2} and 
$$
\begin{cases}
\lim_{t\rightarrow\infty}e^{\lambda_pt}w(\cdot,t)=f\;\;\text{in}\; {\cal C}^\perp\\
\lim_{t\rightarrow\infty}\Phi(e^{\lambda_pt}w(\cdot,t))=\Phi(f).
\end{cases}
$$
Moreover, if $f\neq 0\in {\cal C}^\perp$, then 
$$
\lim_{t\rightarrow\infty}\frac{\Phi(w(\cdot,t))}{\Psi(w(\cdot,t))}=\mu_p.
$$
Consequently, the limit \eqref{scaledLimNBC} holds for $u:=|f|^{q-2}f$, which is necessarily 
an extremal of  \eqref{PoincareNBC}; and if $u\not\equiv 0$, then
we can also conclude \eqref{RatioLimNBC}.  
\end{proof}
\begin{rem} If we do not make the assumption that any two extremals of \eqref{PoincareNBC} are linearly dependent, our methods give that there is a sequence of positive numbers $(t_k)_{k\in\N}$ increasing to 
infinity for which the limit $u:=\lim_{k\rightarrow\infty}e^{\left(\lambda_p/(p-1)\right)t_k}v(\cdot,t_k)$ exists and is extremal for \eqref{NBCDualPoincare}.  If $u\not\equiv 0$, then \eqref{RatioLimNBC} still holds. 
\end{rem}

\section{Fractional Trudinger equation}\label{NonlocalSect}
In this final section, we will study the initial value problem \eqref{TruFrac}. 
Recall that this problem involves the fractional $p-$Laplacian  $(-\Delta_p)^s$ \eqref{DeltaSpSmooth} and a Poincar\'e inequality \eqref{FracPoincare} on the fractional Sobolev space $W^{s,p}_0(\Omega)$.  It is known that extremal functions $u\in W^{s,p}_0(\Omega)$ of \eqref{FracPoincare} exist and satisfy    
$$
\begin{cases}
(-\Delta_p)^su=\lambda_p|u|^{p-2}u\;\; & \text{in}\;\Omega\\
\hspace{.53in} u=0 \;\; &\text{on}\;\ \R^n\setminus\Omega.
\end{cases}
$$
Here $\lambda_p$ is the optimal constant in \eqref{FracPoincare}.  Moreover, extremals have a definite sign in $\Omega$ and the ratio of any two nonvanishing extremals is constant  \cite{LindgrenLindqvist}.

\par We will also define the operator $(-\Delta_p)^s$ more generally as the mapping $(-\Delta_p)^s: W^{s,p}_0(\Omega)\rightarrow (W^{s,p}_0(\Omega))^*$ given by
$$
\langle(-\Delta_p)^su,\psi\rangle:=\iint_{\R^n\times\R^n}\frac{|u(x)-u(y)|^{p-2}(u(x)-u(y))(\psi(x)-\psi(y))}{|x-y|^{n+ps}}dxdy
$$
for $u,\psi\in W^{s,p}_0(\Omega)$. We leave it as an exercise to check that $(-\Delta_p)^s$ is bijective and 
\begin{equation}\label{FracIsometry}
\|(-\Delta_p)^su\|_{(W^{s,p}_0(\Omega))^*}=\|u\|_{W^{s,p}_0(\Omega)}^{p-1},
\end{equation}
where 
$$
\|u\|_{W^{s,p}_0(\Omega)}:=\left(\iint_{\R^n\times\R^n}\frac{|u(x)-u(y)|^p}{|x-y|^{n+ps}}dxdy\right)^{1/p}.
$$
The identity \eqref{FracIsometry} can be used to verify the following dual Poincar\'e inequality 
\begin{equation}\label{dualFracPoincare}
\mu_p\|f\|^q_{(W^{s,p}_0(\Omega))^*}\le \int_\Omega|f|^qdx\quad (f\in L^q(\Omega))
\end{equation}
where $\mu_p:=\lambda_p^{\frac{1}{p-1}}$ (see Appendix \ref{AppDPI}).

\par For a smooth solution $v$ of \eqref{TruFrac} and $t>0$, 
\begin{equation}\label{FracLpGoDown}
\frac{d}{dt}\frac{1}{p}\int_\Omega|v(x,t)|^pdx=-\frac{1}{p-1}\iint_{\R^n\times\R^n}\frac{|v(x,t)-v(y,t)|^p}{|x-y|^{n+ps}}dxdy.
\end{equation}
Consequently, for all $t\ge0$,
$$
\frac{1}{p}\int_\Omega|v(x,t)|^pdx+\frac{1}{p-1}\int^t_0\left\{
\iint_{\R^n\times\R^n}\frac{|v(x,\tau)-v(y,\tau)|^p}{|x-y|^{n+ps}}dxdy\right\}d\tau
=\frac{1}{p}\int_\Omega|g(x)|^pdx.
$$
This observation inspires the following definition. 
\begin{defn}
Assume $g\in L^p(\Omega)$. A {\it weak solution} of \eqref{TruFrac} is a function $v:\R^n\times[0,\infty)\rightarrow \R$ that satisfies:
$(i)$
$$
v\in L^\infty([0,\infty);L^p(\Omega))\cap L^p([0,\infty); W^{s,p}_0(\Omega));
$$
$(ii)$
\begin{align*}
\int^\infty_0\int_\Omega |v|^{p-2}v\psi_t dxdt = \hspace{4in} \\
\int^\infty_0\left\{\iint_{\R^n\times\R^n}\frac{|v(x,t)-v(y,t)|^{p-2}(v(x,t)-v(y,t))(\psi(x,t)-\psi(y,t))}{|x-y|^{n+ps}}dxdy\right\}dt
\end{align*}
for each $\psi\in C^\infty_c(\Omega\times(0,\infty))$; and $(iii)$
$$
v(\cdot,0)=g.
$$
\end{defn}
Some useful continuity and monotonicity properties of weak solutions of \eqref{TruFrac} are listed below. 
\begin{prop}
Assume that $v$ is a weak solution of \eqref{TruFrac}.  Then $v$ has the following properties.
\begin{enumerate}[(i)]

\item $|v|^{p-2}v\in AC^q_{\text{loc}}([0,\infty); (W^{s,p}_0(\Omega))^*)$.

\item $[0,\infty) \ni t\mapsto \displaystyle\int_\Omega|v(x,t)|^pdx$ is locally absolutely continous.

\item \eqref{FracLpGoDown} holds for almost every $t>0$.

\item $[0,\infty) \ni t\mapsto e^{p\left(\frac{\lambda_p}{p-1}\right)t}\displaystyle\int_\Omega|v(x,t)|^pdx$ is nonincreasing.

\item $v:[0,\infty)\rightarrow L^p(\Omega)$ is bounded and uniformly continuous.

\item If $v(\cdot,t)\not\equiv 0$ for $t\ge 0$, then 
$$
[0,\infty) \ni t\mapsto\frac{\||v(\cdot,t)|^{p-2}v(\cdot,t)\|^q_{L^{q}(\Omega)}}{\displaystyle\||v(\cdot,t)|^{p-2}v(\cdot,t)\|^q_{ (W^{s,p}_0(\Omega))^*}}
$$
is nonincreasing. 
\end{enumerate}
\end{prop}

\begin{rem}
As in Remark \ref{StrongMonRemark}, we can verify
$$
\frac{d}{dt}\left\{\frac{\displaystyle\iint_{\R^n\times\R^n}\frac{|v(x,t)-v(y,t)|^p}{|x-y|^{n+ps}}dxdy}{\displaystyle\int_\Omega|v(x,t)|^pdx}\right\}\le 0
$$
for any smooth solution $v$ of \eqref{TruFrac} that doesn't vanish identically.
\end{rem}

\par Weak solutions of \eqref{TruFrac} have compactness properties similar to those detailed in Proposition \ref{CompactTrudinger}. In order to write a corresponding statement for weak solutions \eqref{TruFrac}, we would only need
to change $W^{1,p}_0(\Omega)$ to $W^{s,p}_0(\Omega)$ and change $W^{-1,q}(\Omega)$ to $(W^{s,p}_0(\Omega))^*.$ Moreover, a weak solution can be constructed using the following implicit time scheme: $u^0=g$,
$$
\begin{cases}
\displaystyle\frac{{\cal J}_p(u^k)- {\cal J}_p(u^{k-1})}{\tau}+(-\Delta_p)^su^k=0\;\; &\text{in}\;\Omega \\
\hspace{2.05in}u^k=0\;\;&\text{on}\;\R^n\setminus\Omega
\end{cases}
$$
for $k\in \N$. Using the ideas in our proof of Proposition \ref{ExistenceProp}, we have that there is a weak solution $v$ of \eqref{TruFrac} with 
$$
\iint_{\R^n\times\R^n}\frac{|v(x,t)-v(y,t)|^p}{|x-y|^{n+ps}}dxdy\le
\iint_{\R^n\times\R^n}\frac{|v(x,s)-v(y,s)|^p}{|x-y|^{n+ps}}dxdy
$$
for almost every $t,s\in (0,\infty)$ with $t\ge s$.

\par Employing the results above with the methods used to prove Theorem \ref{TruThm1}, we have the following assertion regarding the large time behavior of weak solutions of \eqref{TruFrac}.  

\begin{customthm}{4}\label{TruThmFrac} (i) Assume $v$ is a weak solution of \eqref{TruFrac}.  Then the limit 
\begin{equation}\label{FracscaledLim}
u:=\lim_{t\rightarrow\infty}e^{\left(\frac{\lambda_p}{p-1}\right)t}v(\cdot,t)
\end{equation}
exists in $L^p(\Omega)$, and $u$ is extremal for \eqref{FracPoincare}. If $u\not\equiv 0$, then $v(\cdot,t)\not\equiv 0$ for $t\ge 0$ and
$$
\mu_p=\lim_{t\rightarrow\infty}\frac{\||v(\cdot,t)|^{p-2}v(\cdot,t)\|^q_{L^q(\Omega)}}{\||v(\cdot,t)|^{p-2}v(\cdot,t)\|^q_{ (W^{s,p}_0(\Omega))^*}}.
$$ 
(ii) There is a weak solution $v$ of \eqref{TruFrac} such that the limit \eqref{FracscaledLim}
exists in $W^{s,p}_0(\Omega)$.  If $u\not\equiv 0$,  
$$
\lambda_p=\lim_{t\rightarrow\infty}\frac{\displaystyle\iint_{\R^n\times\R^n}\frac{|v(x,t)-v(y,t)|^p}{|x-y|^{n+ps}}dxdy}{\displaystyle\int_\Omega|v(x,t)|^pdx}.
$$ 
\end{customthm}

\appendix 

\section{Dual Poincar\'e inequality}\label{AppDPI}
This section is devoted to deriving the dual Poincar\'e inequality \eqref{dualPoincare} and characterizing its equality condition.  Analogous computations can be used to establish inequalities \eqref{RBCdualPoincare}, \eqref{NBCDualPoincare}, and \eqref{dualFracPoincare} their respective equality conditions.

\par Let $f\in W^{-1,q}(\Omega)$, and choose $u\in W^{1,p}_0(\Omega)$ so that $-\Delta_pu=f$; here the $p$-Laplacian $-\Delta_p: W^{1,p}_0(\Omega)\rightarrow W^{-1,q}(\Omega)$ is defined by the formula \eqref{W1pzeroPlaplace}. We have by \eqref{pLapIsometry} that
\begin{align*}
\langle f, u\rangle&=\langle -\Delta_pu, u\rangle\\
&=\|u\|^p_{W^{1,p}_0(\Omega)}\\
&=\|f\|^q_{W^{-1,q}(\Omega)}.
\end{align*}
\par Now suppose in addition that $f\in L^q(\Omega)$ and $f\not\equiv 0$.  The above computation, H\"older's inequality and \eqref{Poincare} together imply
\begin{align*}
\|f\|^q_{W^{-1,q}(\Omega)}&=\int_{\Omega}fudx\\
&\le \left(\int_\Omega|f|^qdx\right)^{1/q}\left(\int_\Omega|u|^pdx\right)^{1/p}\\
&\le \left(\int_\Omega|f|^qdx\right)^{1/q}\frac{1}{\lambda_p^{1/p}}\left(\int_\Omega|Du|^pdx\right)^{1/p}\\
&= \left(\int_\Omega|f|^qdx\right)^{1/q}\frac{1}{\lambda_p^{1/p}}\|u\|_{W^{1,p}_0(\Omega)}\\
&= \left(\int_\Omega|f|^qdx\right)^{1/q}\frac{1}{\lambda_p^{1/p}}\|f\|^{q-1}_{W^{-1,q}(\Omega)}.
\end{align*}
Setting $\mu_p=\lambda_p^{\frac{1}{p-1}}$, we then have
$$
\mu_p\|f\|^q_{W^{-1,q}(\Omega)}\le\int_\Omega|f|^qdx,
$$
which is the dual Poincar\'e inequality \eqref{dualPoincare}.

\par If equality holds in our computations above, then $u$ is extremal for the Poincar\'e inequality \eqref{Poincare} and equality holds in our application of H\"older's inequality. Consequently, 
$$
\frac{f}{\|f\|_{L^q(\Omega)}}=\frac{|u|^{p-2}u}{\|u\|^{p-1}_{L^p(\Omega)}}.
$$
Therefore, equality holds in the dual Poincar\'e inequality  \eqref{dualPoincare} if and only if $f=|u_0|^{p-2}u_0$ for an extremal $u_0$ of the Poincar\'e inequality \eqref{Poincare}.

\begin{rem}
There is another way to derive  \eqref{dualPoincare}. The Poincar\'e inequality \eqref{Poincare} expresses that $W^{1,p}_0(\Omega)\subset L^p(\Omega)$ with the continuous embedding $i: W^{1,p}_0(\Omega)\rightarrow L^p(\Omega); u\mapsto u$. In particular, in \eqref{Poincare} we have $\|i\|=\lambda_p^{-1/p}.$  It follows that the adjoint operator $i^*: L^q(\Omega)\rightarrow W^{-1,q}(\Omega)$ is the continuous embedding of $L^q(\Omega)\subset W^{-1,q}(\Omega)$  and $\|i^*\|=\lambda_p^{-1/p}$. We then conclude \eqref{dualPoincare}.
\end{rem}

\section{Norm on ${\cal C}^\perp$}\label{normX}
Recall the definitions of the space ${\cal C}^\perp$ \eqref{Xspace}, the space ${\cal S}$ \eqref{NeumannSpaceS}, the operator ${\cal A}_p$ \eqref{NeumannPLaclace}, and the function $\|\cdot \|_{{\cal C}^\perp}$ \eqref{normXdefn}.  We will show 
that $\|\cdot\|_{{\cal C}^\perp}$ is a norm on ${\cal C}^\perp$ and that ${\cal C}^\perp$ is a reflexive Banach space under this norm.  
The dual Poincar\'e inequality \eqref{NBCDualPoincare} follows from the arguments given in  Appendix \ref{AppDPI} once we observe that for any $f\in {\cal C}^\perp$ and  corresponding weak solution $u$ of \eqref{NeumannProblem} that satisfies \eqref{pAverageZero}, 
\begin{equation}\label{XnormToTheQ}
\|f\|_{{\cal C}^\perp}^q=\langle f, u\rangle =\int_\Omega|Du|^pdx.
\end{equation}

\begin{prop}
$\|\cdot\|$ is a norm on ${\cal C}^\perp$. 
\end{prop}
\begin{proof}
\par By formula \eqref{XnormToTheQ}, $\|f\|_{{\cal C}^\perp}\ge 0$ and if $\|f\|_{{\cal C}^\perp}=0$, then $f=0\in {\cal C}^\perp$. Moreover, ${\cal A}_p$ degree $p-1$ homogeneous, so its inverse is  degree $q-1$ homogeneous. It follows that $\|\cdot \|_{{\cal C}^\perp}$ is positively homogeneous.  Thus, we are left to argue that the triangle inequality holds. 

\par To this end, we first assert 
\begin{equation}\label{XMinkowski}
\langle f, {\cal A}_p^{-1}g\rangle \le \|f\|_{{\cal C}^\perp}\|g\|_{{\cal C}^\perp}^{q-1}
\end{equation}
for every $f,g\in {\cal C}^\perp$.  In order to verify this claim, we choose $u, v\in {\cal S}$ and that solve ${\cal A}_pu=f$ and ${\cal A}_pv=g$.  Then we have
\begin{align*}
\langle f, {\cal A}_p^{-1}g\rangle &= \int_\Omega |Du|^{p-2}Du\cdot Dvdx\\
&\le \left(\int_\Omega |Du|^{p}dx\right)^{1-1/p}\left(\int_\Omega |Dv|^{p}dx\right)^{1/p}\\
&= \langle f,u\rangle^{1/q}\langle g,v\rangle^{1-1/q}\\
&= \|f\|_{{\cal C}^\perp}\|g\|_{{\cal C}^\perp}^{q-1}.
\end{align*}

\par Now let $f_1, f_2\in {\cal C}^\perp$. Using \eqref{XMinkowski} and Young's inequality, we compute 
\begin{align*}
\|f_1+f_2\|_{{\cal C}^\perp}^q&=\langle f_1+f_2, {\cal A}_p^{-1}\left(f_1+f_2\right)\rangle\\
&=\langle f_1, {\cal A}_p^{-1}\left(f_1+f_2\right)\rangle +\langle f_2, {\cal A}_p^{-1}\left(f_1+f_2\right)\rangle\\
&\le \|f_1\|_{{\cal C}^\perp}\|f_1+f_2\|_{{\cal C}^\perp}^{q-1}+\|f_2\|_{{\cal C}^\perp}\|f_1+f_2\|_{{\cal C}^\perp}^{q-1}\\
&=\left(\|f_1\|_{{\cal C}^\perp}+\|f_2\|_{{\cal C}^\perp}\right)\|f_1+f_2\|_{{\cal C}^\perp}^{q-1}\\
&\le \frac{1}{q}\left(\|f_1\|_{{\cal C}^\perp}+\|f_2\|_{{\cal C}^\perp}\right)^q +\left(1-\frac{1}{q}\right)\|f_1+f_2\|_{{\cal C}^\perp}^q.
\end{align*}
Therefore,
$$
\|f_1+f_2\|_{{\cal C}^\perp}\le \|f_1\|_{{\cal C}^\perp}+\|f_2\|_{{\cal C}^\perp}.
$$
\end{proof}
\begin{rem}\label{SubDiffRem}
For $f,g\in C^\perp$, 
\begin{align*}
\langle f-g, A_p^{-1}g\rangle&=\langle f, A_p^{-1}g\rangle-\langle g, A_p^{-1}g\rangle\\
&\le  \|f\|_{{\cal C}^\perp}\|g\|_{{\cal C}^\perp}^{q-1}-\|g\|^q_{{\cal C}^\perp}\\
&\le \frac{1}{q}\|f\|^q_{{\cal C}^\perp}+\left(1-\frac{1}{q}\right)\|g\|^q_{{\cal C}^\perp}-\|g\|^q_{{\cal C}^\perp}\\
&\le \frac{1}{q}\|f\|^q_{{\cal C}^\perp}-\frac{1}{q}\|g\|^q_{{\cal C}^\perp}.
\end{align*}
Consequently, $A_p^{-1}g$ belongs to the subdifferential of $\frac{1}{q}\|\cdot\|^q_{{\cal C}^\perp}$ at $g$. 
\end{rem}

In order to conclude ${\cal C}^\perp$ is a Banach space, it suffices to show that the norm $\|\cdot\|_{{\cal C}^\perp}$ is equivalent to the standard norm 
\begin{equation}\label{StandardNormonX}
\|f\|_{(W^{1,p}(\Omega))^*}:=\sup\left\{\langle f,\phi\rangle: \|\phi\|_{W^{1,p}(\Omega)}\le 1\right\}
\end{equation}
on $W^{1,p}(\Omega)$. Here $\|\phi\|_{W^{1,p}(\Omega)}:=\left(\int_\Omega(|\phi|^p+|D\phi|^p)dx\right)^{1/p}$.  Below, $\lambda_p$ is the same constant appearing in the Poincar\'e inequality \eqref{PoincareNBC}. 
\begin{prop}
For each $f\in {\cal C}^\perp$, 
$$
\frac{1}{\left(1+1/\lambda_p\right)^{1/p}}\|f\|_{{\cal C}^\perp}\le \|f\|_{(W^{1,p}(\Omega))^*}\le \|f\|_{{\cal C}^\perp}.
$$
\end{prop}
\begin{proof}
Let $f\neq 0\in {\cal C}^\perp$, and select the weak solution $u\in W^{1,p}(\Omega)$ of \eqref{NeumannProblem} that satisfies \eqref{pAverageZero}. By \eqref{WeakSolnNeumann}, 
\begin{align*}
\langle f,\phi\rangle &\le \left(\int_\Omega |Du|^{p}dx\right)^{1-1/p}\left(\int_\Omega |D\phi|^{p}dx\right)^{1/p}\\
&\le \left(\int_\Omega |Du|^{p}dx\right)^{1-1/p}\|\phi\|_{W^{1,p}(\Omega)}\\
&= \|f\|_{{\cal C}^\perp}\|\phi\|_{W^{1,p}(\Omega)}.
\end{align*}
Thus, $\|f\|_{(W^{1,p}(\Omega))^*}\le \|f\|_{{\cal C}^\perp}$. 

\par Conversely, we can employ \eqref{XnormToTheQ} and the Poincar\'e inequality \eqref{PoincareNBC} to find
\begin{align*}
\|f\|_{(W^{1,p}(\Omega))^*}&\ge \left\langle f , \frac{u}{\|u\|_{W^{1,p}(\Omega)}} \right\rangle \\
&=\frac{\langle f , u\rangle}{\|u\|_{W^{1,p}(\Omega)}}\\
&=\frac{\int_\Omega|Du|^pdx}{\|u\|_{W^{1,p}(\Omega)}}\\
&=\frac{\left(\int_\Omega|Du|^pdx\right)^{1-1/p}\left(\int_\Omega|Du|^pdx\right)^{1/p}}{\|u\|_{W^{1,p}(\Omega)}}\\
&=\|f\|_{{\cal C}^\perp}\frac{\left(\int_\Omega|Du|^pdx\right)^{1/p}}{\left(\int_\Omega(|u|^p+|Du|^p)dx\right)^{1/p}}\\
&\ge \|f\|_{{\cal C}^\perp} \frac{1}{\left(1+1/\lambda_p\right)^{1/p}}.
\end{align*}
\end{proof}
Let us finally argue that ${\cal C}^\perp$ is reflexive.  Using the definition \eqref{Xspace}, it is possible to show that ${\cal C}^\perp$ equipped with the standard norm \eqref{StandardNormonX} is isometrically isomorphic to
$$
\left(W^{1,p}(\Omega)/{\cal C}\right)^*
$$
(see chapter 5, exercise 23 of \cite{folland} for details). As ${\cal C}$ is a closed subspace of the reflexive space $W^{1,p}(\Omega)$, then $W^{1,p}(\Omega)/{\cal C}$ is also reflexive. It then follows that $\left(W^{1,p}(\Omega)/{\cal C}\right)^*$ is necessarily reflexive. Consequently, ${\cal C}^\perp$ equipped with the standard norm \eqref{StandardNormonX} is reflexive.  Since $\|\cdot\|_{{\cal C}^\perp}$ is an equivalent norm to \eqref{StandardNormonX}, ${\cal C}^\perp$ is reflexive when equipped with $\|\cdot\|_{{\cal C}^\perp}$, as well.


\end{document}